\documentclass{article}
\usepackage[capitalise]{cleveref}
\usepackage{amsthm}
\usepackage{amsmath}
\usepackage{amssymb}
\usepackage{amsfonts}
\usepackage[dvips]{graphicx}
\usepackage{caption}
\usepackage[numbers,sort&compress]{natbib}

\title{The Rational Chebyshev of Second Kind Collocation Method for Solving a Class of Astrophysics Problems}
\author {K. Parand$^{\mbox{\footnotesize
a}}$\thanks{E-mail address: \texttt{k\_parand@sbu.ac.ir},
Corresponding author, (K. Parand)}\hspace{1mm}
, \normalsize{S. Khaleqi$^{\mbox{\footnotesize
a}}$\thanks{E-mail address: \texttt{sajjad@gmail.com},
(S. Khaleghi)}\hspace{1mm}}
\vspace{-2mm}\\\\
\footnotesize{\em $^{\mbox{\footnotesize a}}$Department$\!$ of$\!$
Mathematics, Faculty$\!$ of$\!$ Mathematical Sciences$\!$,}\vspace{-2mm}\\
\footnotesize{\em Shahid Beheshti University, Evin,Tehran 19839,Iran}\\
}
\begin{document}
\newtheorem{theorem}{Theorem}
\maketitle
\begin{abstract}
The Lane-Emden equation has been used to model several phenomenas in theoretical physics, mathematical physics and astrophysics such as the theory of stellar structure. This study is an attempt to utilize the collocation method with the Rational Chebyshev of Second Kind function (RSC) to solve the Lane-Emden equation over the semi-infinit interval $[0,+\infty)$. According to well-known results and comparing with previous methods, it can be said that this method is efficient and applicable.
\end{abstract}
\begin{keywords}
Lane-Emden,Collocation method,Rational Chebyshev of Second Kind,Nonlinear ODE,Astrophysics
\end{keywords}
\newpage
\section{introduction}
Solving science and engineering problems appeared in unbounded domains, scientists have studied about it in the last decade.
Spectral method is one of the famous solution. Some researchers have proposed several spectral method to solve these kinds of problems. Using spectral method related to orthogonal systems in unbounded domains such as the Hermite spectral method and the Laguerre Method is called direct approximation ~\cite{could 90,funaro 91,funaro 90,gua 99,gua 2000a,mady 85,shen 2000,siyyam 2001}.\\
In an indirect approximation, the original problem mapping in an unbounded domain is changed to a problem in a bounded domain. Gua has used this method by choosing a suitable Jacobi polynomials to approximate the results ~\cite{gua 98,gua 2000b,gua 2000c}. There is Another method called domain truncation. In this method, the [0,$+\infty$) to [0,L] and ($-\infty$,$+\infty$) to [-L,+L] are replaced by choosing a sufficient large L ~\cite{boyd 2000}.
Using spectral method based on the rational approximation is one of the direct approximation has been used in this paper.
 Authors of~\cite{christov 1982,boyd 1987a,boyd 1987b} have developed some spectral methods on unbounded intervals by using some mutually orthogonal systems of rational functions. Boyd ~\cite{boyd 1987b} has proposed a new spectral basis named Rational Chebyshev functions on the semi-infinite interval by mapping to the Chebyshev polynomials. Author of ~\cite{gushen 2000} has suggested and discussed a set of Legendre rational functions which are mutually orthogonal in $L^{2}(0, +\infty)$ with a non-uniform weight function $\omega(x) = (x + 1)^{-2}$. Authors of~\cite{paraz 2004a,paraz 2004b,paraz 2004c,parnd shahini,paraz taghavi} have applied some Rational Chebyshev and Legendre functions by tau and collocation method which were used to solve nonlinear ordinary differential equations on semi-infinite intervals. Author of~\cite{paraz 2004a,paraz 2004b,paraz 2004c} have obtained the operational matrices of the derivative and product of rational Chebyshev and Legendre functions to reduce the solution of these problems to the solution of some algebraic equations.\\

 The Lane-Emden equation has been used to model several phenomenas in theoretical physics, mathematical physics and astrophysics as the theory of stellar structure. It appears in other contexts, for example in case of radiatively cooling, self-gravitating gas clouds, in the mean-field treatment of a phase transition in critical absorption or in the modeling of clusters of galaxies. It has been classified as singular initial value problems in the second order ordinary differential equations. This equation is to honor astrophysicists Jonathan Homer Lane and Robert Emden which is called Lane-Emden.~\cite{bender 1986,mandel 2001,shawag 1993,ramos 2005,yousefi 2006,davis 2006}\\
 Consider the Poisson equation and the condition for hydrostatic equilibrium~\cite{chandr 1967,46,dehghan 2008,agarwal 2007}:
 \begin{equation}
 \frac{{dP}}{{dr}} =  - \rho \frac{{GM(r)}}{{{r^2}}},
 \end{equation}

 \begin{equation}
 \frac{{dM(r)}}{{dr}} = 4\pi \rho {r^2},
 \end{equation}
 where $G$ is the gravitational constant, $P$ is the pressure, $M(r)$ is
the mass of a star at a certain radius $r$, and $\rho$ is the density, at a
distance $r$ from the center of a spherical star~\cite{chandr 1967}.The combination of these equations yields the following equation, which as should be noted, is an equivalent form of the Poisson equation~\cite{chandr 1967,46,dehghan 2008,agarwal 2007}:

\begin{equation}
 \frac{1}{{{r^2}}}\frac{d}{{dr}}\left( {\frac{{{r^2}}}{\rho }\frac{{dP}}{{dr}}} \right) =  - 4\pi G\rho .
\end{equation}

From these equations one can obtain the Lane-Emden type equation through the simple assumption that the density is simply
related to the density, while remaining independent of the temperature. We already know that in the case of a degenerate electron
gas, the pressure and density are$~\rho  \sim {P^{\frac{3}{5}}}$, assuming that such a relation exists for other states of the star, we are led to consider a relation of the following form~\cite{chandr 1967}:

\begin{equation}
P = K{\rho ^{1 + \frac{1}{m}}},
\end{equation}

where $K$ and $m$ are constants and $m$ is the polytropic index which is related to the ratio of specific heats of the gas comprising the star.Based upon these
assumptions we can insert this relation into our first equation for
the hydrostatic equilibrium condition and from this equation ~\cite{chandr 1967,46,dehghan 2008,agarwal 2007} we have
\begin{equation}
\left[ {\frac{{K\left( {m + 1} \right)}}{{4\pi G}}{\lambda ^{\frac{1}{m} - 1}}} \right]\frac{1}{{{r^2}}}\frac{d}{{dr}}\left( {{r^2}\frac{{dy}}{{dr}}} \right) =  - {y^m},
\end{equation}
where the additional alteration to the expression for density has
been inserted with $\lambda$ representing the central density of the star
and $y$ that of a related dimensionless quantity that are both related to $\rho$ through the following relation~\cite{chandr 1967,46,dehghan 2008,agarwal 2007}
\begin{equation}
\rho  = \lambda {y^m},
\end{equation}
Additionally, if place this result into the Poisson equation, we obtain a differential equation for the mass, with a dependence upon the polytropic index $m$. Though the differential equation is seemingly difficult to solve, this problem can be partially alleviated by the introduction of an additional dimensionless variable $x$ ,given by the following:
\begin{equation}
r = ax
\end{equation}
\begin{equation}
a = {\left[ {\frac{{K\left( {m + 1} \right)}}{{4\pi G}}{\lambda ^{\frac{1}{m} - 1}}} \right]^{\frac{1}{2}}}
\end{equation}
Inserting these relations into our previous equations we obtain the
famous form of the Lane-Emden type equations, given in the following:
\begin{equation}
\frac{1}{{{x^2}}}\frac{d}{{dx}}\left( {{x^2}\frac{{dy}}{{dx}}} \right) =  - {y^m}
\end{equation}
Taking these simple relations we will have the standard Lane-Emden equation with $g(y)=y^m$ ~\cite{chandr 1967,46,dehghan 2008,agarwal 2007},
\begin{equation}
y'' + \frac{2}{x}y' + {y^m} = 0,~~~x > 0.\label{standard}
\end{equation}
At this point it is also important to introduce the boundary conditions which are based upon the following boundary conditions
for hydrostatic equilibrium and normalization consideration of the
newly introduced quantities $x$ and $y$. What follows for $r=0$ is
\begin{equation}
r = 0 \to x = 0,\rho  = \lambda  \to y\left( 0 \right) = 1.\label{bond}
\end{equation}
As a result an additional condition must be introduced in order to
maintain the condition of equation\eqref{bond} simultaneously:
\begin{equation}
y'=0
\end{equation}
In other words, the boundary conditions are as follows
\begin{equation}
y\left( 0 \right) = 1,y'\left( 0 \right) = 0
\end{equation}
The values of $m$ which are physically interesting, lie in the interval
[0,5]. The main difficulty in the analysis of this type of equation is
the singularity behavior occurring at $x=0$.
Exact solutions for equation \eqref{standard} are known only for$m=0,1$ and $5$. For other values of $m$ the standard Lane-Emden equation is to
be integrated numerically. Thus we decided to present a new and efficient technique to solve it numerically.\\
This paper is organized as follows: In Section 2 we survey several methods that have been used to
solve Lane-Emden type equations. In Section 3, the properties of
Chebyshev Function of the Second Kind and the way to construct the collocation technique for this type of equation are described. In Section 4 the proposed method is applied to some types of Lane-Emden equations, and a comparison is made with the existing analytic or exact solutions that were reported in other published works in the literature. Finally we give a brief conclusion in the last section.

\section{The Methods have been proposed to solve Lane-Emden type equations}
Many analytical methods have been applied to solve Lane-Emden type equations. The major problem has been raised in the singularity of the equations at $x=0$.\\

Bender et al.\cite{bender 1986} have shown the perturbative technique by using $\delta$-method has given excellent result when applied to difficult non-linear differential equation such as Lane-Emden.\\

Kara et al.\cite{kara92} have shown the Lagrangian method to solve Emden equation. They have categorized in non-linearzable equations and one year later \cite{kara93} they have found  general  solutions  of  Lane-Emden equation  by  determining  the  Lie  point  symmetries and,  hence,  the Lie algebra  it generates.\\

Shawagfeh\cite{shawag 1993} has approached  the  problem  differently  by  utilizing  the  newly  developed  Adomian  decomposition  method to solve Lane-Emden equation and has  used  this  method  to  derive  a   nonperturbative  approximate solution  in  the  form  of  power  series  which  is  easily  computable  and  accurate. To accelerate the convergence he has used the Pade` approximates method.\\

Mandelzweig et al.\cite{mandel 2001} have shown that the quasilinearization method (QLM) gives excellent results when applied to different nonlinear ordinary differential equations in physics, such as Lane-Emden equation.\\

Wazwaz \cite{29} has employed a reliable algorithm rests mainly on the Adomian decomposition method with an alternate framework designed to overcome the difficulty of the singular point and he has proposed framework has been applied to a generalization of Lane-Emden equations.

Liao \cite{33} has introduced an analytic algorithm logically contains the well-known Adomian decomposition method, provides for Lane-Emden equation. This algorithm itself provides us with a convenient way to adjust convergence regions even without Pade technique.

He \cite{34} has presented by the semi-inverse method, a variational rinciple is obtained for the Lane-Emden equation, which gives much numerical convenience when applying finite element methods or Ritz method.

Ramos \cite{35} has uttered Linearization methods provide piecewise linear ordinary differential equations which can be easily integrated, and provide accurate answers even for hypersingular potentials, for which perturbation methods diverge and he has shown the applicability and accuracy of linearization methods for initial-value problems in ordinary differential equations are verified on examples that include the Lane-Emden equation, and some other equations.

Parand et al. \cite{20} has expressed numerical method based on rational Legendre tau for solving the Lane-Emden nonlinear differential equation. They have provided the operational matrices of derivative and product of rational Legendre functions.

Ramos \cite{36} has employed Linearization methods result in linear constant-coefficients ordinary differential equations which can be integrated analytically, thus yielding piecewise analytical solutions and globally smooth solutions.

Wazwaz \cite{32} has shown the modified decomposition method which has been applied for analytic treatment of nonlinear differential equations such as Lane-Emden. This method accelerates the rapid convergence of the series solution, dramatically reduces the size of work, and provides the exact solution by using few iterations only without any need to the so-called Adomian polynomials.

Yusefi \cite{yousefi 2006} has presented a numerical method by using integral operator and convert Lane-Emden equations to integral equations and then applying Legendre wavelet approximations.

Ramos \cite{37} has detailed series solutions of the Lane-Emden equation based on either a Volterra integral equation formulation or the expansion of the dependent variable in the original ordinary differential equation and has compared with series solutions obtained by means of integral or differential equations based on a transformation of the dependent variables.

Aslanov \cite{46} has constructed a recurrence relation for the components of the approximate solution and has investigated the convergence conditions of the Emden-Fowler type of equations.

Razzaghi et al \cite{48} have given numerical method based upon hybrid function approximations for solving nonlinear initial-value problems such as Lane-Emden type equations.The properties of hybrid of block-pulse functions and Lagrange interpolating polynomials have presented and and have utilized to reduce the computation of nonlinear initial-value problems to a system of non-algebraic equations.

Dehghan et al.\cite{55} have investigated The Lane-Emden using the variational iteration method and it have been shown the efficiency and applicability of this procedure for solving this equation.

Hashim et al. \cite{42} have obtained approximate and/or exact analytical solutions of the generalized Emden-Fowler type equations in the second-order ordinary differential equations (ODEs) by homotopy-perturbation method (HPM).

Parand et al \cite{23} have represented a pseudospectral technique is to solve the Lane-Emden type equations on a semi-infinite domain. The method is based on rational Legendre functions and Gauss-Radau integration.

Ramos \cite{38} has shown that the appearance of noise terms in the decomposition method is related to both the differential equation and the manner in which the homotopy parameter is introduced, especially for the Lane-Emden equation.

Bataineh et al \cite{40} have stated approximate and/or exact analytical solutions of singular initial value problems (IVPs) of the Emden-Fowler type in the second-order ordinary differential equations (ODEs) are obtained by the homotopy analysis method (HAM).

\"{O}zi\c{s} et al.\cite{47} have declared, approximate-exact solutions of a class of Lane-Emden type singular IVPs problems, by the variational iteration method.

Singh et al. \cite{49} have shown an efficient analytic algorithm for Lane-Emden type equations using modified homotopy analysis method, which is different from other analytic techniques as it itself provides us with a convenient way to adjust convergence regions even without Pade technique.

Parand et al.\cite{-1} have provide a pseudo spectral method for Lane-Emden equation which based on some orthogonal functions.

Parand et al.\cite{-3} have imparted a method based on the rational Legendre functions and Gauss-Radau integration to solve Lane-Emden type equations on a semi-infinite domain.

Geng et al.\cite{-4} have investigated the nonlinear singular initial value problems including generalized Lane-Emden-type equations by combining homotopy perturbation method (HPM) and reproducing kernel Hilbert space method (RKHSM).

Karimi et al.\cite{-5} have presented a numerical method which produces an approximate polynomial solution for solving Lane-Emden equations as singular initial value problems. They have used an integral operator to convert Lane-Emden equations into integral equations the convert the acquired integral equation into a power series and finally, transforming the power series into Padé series form, they have obtained an approximate polynomial of arbitrary order to solve Lane-Emden equations.

Gorder \cite{-6} has applied the $\delta$-expansion method to a transformed Lane-Emden equation then the results have been transformed back, and has been recovered analytical solutions to the Lane-Emden equation of the second kind in a special case. The rapid convergence of the method results in qualitatively accurate solutions in relatively few iterations, as we see when we compare the obtained analytical solutions to numerical results.

Wazwaz et al. \cite{-8} have provided a comparison of the Adomian decomposition method (ADM) with the variational iteration method (VIM) for solving the Lane-Emden equations of the first and second kinds.

Yüzbasi \cite{-9} has mentioned a collocation method based on the Bessel polynomials is presented for the approximate solution of a class of the nonlinear Lane-Emden type equations.

Bhrawy et al. \cite{-10} have conveyed a shifted Jacobi pseudo-spectral method JPSM, along with the Boubaker Polynomials Expansion Scheme (BPES), for solving a nonlinear Lane-Emden type equation.

Bhrawy et al. \cite{-11} have related a shifted Jacobi-Gauss collocation spectral method for solving the nonlinear Lane-Emden type equation.

Caglar et al. \cite{-12} have solved a time-dependent heat-like Lane-Emden equation by using a non-polynomial spline method.

Shen \cite{-13} has combined the compactly supported radial basis function (RBF) collocation method and the scaling iterative algorithm to compute and visualize the multiple solutions of the Lane-Emden-Fowler equation on a bounded domain $\Omega\subset R^{2}$ with a homogeneous Dirichlet boundary condition.

Pandey et al. \cite{-14} have developed an efficient numerical method for solving linear and nonlinear Lane-Emden type equations using Legendre operational matrix of differentiation.

Pandey et al. \cite{-15} have applied the First Bernstein operational matrix of differentiation derived using Bernstein polynomials to solve the linear and nonlinear differential equations of Lane-Emden type.

Biezuner et al. \cite{-17} have introduced an iterative method to compute the first eigenpair $(\lambda_{p},e_{p})$ for the $p$-Laplacian operator with homogeneous Dirichlet data as the limit of $(\mu_{q},u_{q})$ as $q\rightarrow p^{-}$, where $u_{q}$ is the positive solution of the sublinear Lane-Emden equation $-\Delta_{p}u_{q}=\mu_{q}u_q^{q - 1}$ with the same boundary data.

Boubaker et al. \cite{-18} have used the Boubaker Polynomials Expansion Scheme (BPES) in order to obtain analytical-numerical solutions to the Lane-Emden initial value problem of the first kind and second kind.

Rismani et al. \cite{-19} have utilized the improved Legendre-spectral method to solve Lane-Emden type equations.

Jalab et al. \cite{-20} have shown a numerical method based on neural network, for solving the Lane-Emden equations singular initial value problems. The numerical solution has been given for integer case and non integer case. The non integer case is taken in the sense of Riemann-Liouville operators.

Wazwaz et al. \cite{-21} have used the systematic Adomian decomposition method to handle the integral form of the Lane-Emden equations with initial values and boundary conditions and have confirmed their belief that the Adomian decomposition method (has provided) provides efficient algorithm for analytic approximate solutions of the equation.

Karimi Vanani et al. \cite{-22} have presented a numerical algorithm based on an operational Tau method (OTM) for solving the Lane-Emden equations as singular initial value problems.

Parand et al \cite{-23} have introduced the Bessel orthogonal functions as new basis for spectral methods and also present an efficient numerical algorithm based on them and collocation method to solve these well-known equations.

Kaur et al. \cite{-25} have provided a technique to investigate the solutions of generalized nonlinear singular Lane-Emden equations of first and second kinds by using a Haar wavelet quasi-linearization approach.

\"{O}zt\"{u}rk et al. \cite{-26} have stated the numerical solution of Lane-Emden equations by using truncated shifted Chebyshev series together with the operational matrix.

Razzaghi \cite{-27} has proposed a numerical method based on hybrid function approximations and Bernoulli polynomials has been presented and utilized to reduce the computation of nonlinear initial-value problems such as Lane-Emden type equations to a system of equations.

Li et al \cite{-28} have computed and visualized multiple solutions of the Lane-Emden equation on a square and a disc, using Legendre and Fourier-Legendre pseudospectral methods based on the Liapunov-Schmidt reduction and symmetry-breaking bifurcation theory.

Wazwaz \cite{-30} has established the Volterra integro-differential forms of the Lane-Emden equations and has used the variational iteration method (VIM) to effectively treat these forms.

Rach et al. \cite{-31} have considered the coupled Lane-Emden boundary value problems in catalytic diffusion reactions by the Adomian decomposition method.

Mohammadzadeh et al. \cite{-32} has shown three numerical techniques based on cubic Hermite spline functions for the solution of Lane-Emden equation.

Wazwaz et al. \cite{-33} have investigated systems of Volterra integral forms of the Lane-Emden equations have used the systematic Adomian decomposition method to handle these systems of integral forms.

\"{O}zt\"{u}rk et al \cite{-34} have shown numerical method depends on collocation method and based on first taking the truncated Hermite series of the solution function, transforms Lane-Emden type equation and given conditions into a matrix equation,solving the system of algebraic equations using collocation points and found the coefficients of the truncated Hermite series.

Wazwaz et al. \cite{-35} have presented new alternate derivations for the Volterra integral forms of the Lane-Emden equation for the shape factor of k=1, where an alternate derivation for L'Hospital's formula will be developed.

G\"{u}rb\"{u}z et al. \cite{-36} have mentioned the numerical method is based on the matrix relations of Laguerre polynomials and their derivatives, and reduces the solution of the Lane-Emden type functional differential equation to the solution of a matrix equation corresponding to system of algebraic equations with the unknown Laguerre coefficients.

\section{Properties of Chebyshev Function of the Second Kind}
The chebyshev polynomials in second kind are defined on $\left[ {- 1,1} \right]$ by
\begin{equation}
{U_n}\left( x \right) = \frac{{\sin \left( {n + 1} \right)\theta }}{{\sin \theta }},x = \cos \theta ,n = 0,1,2, \cdots
\end{equation}
with respect to weight function $w\left( x \right) = \sqrt {1 - {x^2}} $ and can be determined following recurrence formula:
\begin{equation}
{U_0}\left( x \right) = 1,{U_1}\left( x \right) = 2x
                                            \nonumber \\
\end{equation}
\begin{equation}
{U_n}\left( x \right) = 2x{U_{n - 1}}\left( x \right) - {U_{n - 2}}\left( x \right),n = 2,3,4, \cdots \label{rec}
\end{equation}
We can obtain another formula for derivative of $U_n(x)$ as
\begin{equation}
\frac{d}{{dx}}{U_n}(x) = \frac{{(n + 2){U_{n - 1}}(x) - n{U_{n + 1}}(x)}}{{2(1 - {x^2})}} \label{du}
\end{equation}

Since the range $\left[ {0, + \infty } \right)$ quite often more convenient to use than the range $\left[ {-1, +1} \right]$, we sometimes map the independent variable $x$ in $\left[ {0, + \infty } \right)$ to the variable $s$ in $\left[ {-1, +1} \right]$ by the transformation
\begin{equation}
s = \frac{{x - 1}}{{x + 1}} ~~or~~ x = \frac{{s + 1}}{{1 - s}}
\end{equation}
and this leads to a shifted Chebyshev polynomial (of the second kind) $U_n^*$ of degree $n$ in $x$ on $\left[ {0, + \infty } \right)$ given by
\begin{equation}
U_n^*\left( x \right) = {U_n}\left( s \right) = {U_n}\left( {\frac{{x - 1}}{{x + 1}}} \right)\label{shift}
\end{equation}
with respect to weight function ${w^*}(x) = \frac{{4\sqrt x }}{{{{\left( {x + 3} \right)}^3}}}$.
The $U_n^*$ has been called Rational Chebyshev of Second kind (RCS).
From \eqref{shift} and \eqref{rec}, we may deduce the recurrence relation for $U_n^*$ in the form
\begin{equation}
{U_n^*}\left( x \right) = 2\left( {\frac{{x - 1}}{{x + 1}}} \right){U_{n - 1}^*}\left( x \right) - {U_{n - 2}^*}\left( x \right),n = 2,3,4, \cdots \label{shiftrec}
\end{equation}
with initial conditions
\begin{equation}
{U_0^*}\left( x \right) = 1,{U_1^*}\left( x \right) = 2\left( {\frac{{x - 1}}{{x + 1}}} \right)
\end{equation}
As same like \eqref{du} we can obtain derivative of $U^*_n(x)$ as follow
\begin{equation}
\frac{d}{{dx}}{U^*_n}(x) = \frac{2}{{{{(1 + x)}^2}}} \frac{{(n + 2){U^*_{n - 1}}(x) - n{U^*_{n + 1}}(x)}}{{2(1 - {\left( {\frac{{x - 1}}{{x + 1}}} \right)^2})}} \label{dustar}
\end{equation}

\begin{theorem}
In the Gauss–Chebyshev formula
\begin{equation}
\int\limits_{ - 1}^1 {f(x)w(x)dx \simeq \sum\limits_{k = 1}^n {{A_k}f({x_k})} }
\end{equation}
where$\{ x_k\}$ are the $n$ zeros of ${\phi}_n(x)$,the coefficients $A_k$ are as follows:\\
for $w(x)={(1 - x)^{ - \frac{1}{2}}},{\phi}_n(x)={U_n}(x) ~ and ~ A_k=(1-x^2_k)$
\end{theorem}

\begin{proof}
The complete proof has been shown by Mason et al. in \cite{-37}
\end{proof}

Let ${w^*}(x) = \frac{{4\sqrt x }}{{{{\left( {x + 3} \right)}^3}}}$ denote a non-negative, integrable, real-valued function over the interval $I=\left[ {0, + \infty } \right)$. We have defined
\begin{equation}
L_{{w^*}}^2\left( I \right) = \left\{ {v:I \to \mathbb{R}|v{\text{ is measurable and }}{{\left\| v \right\|}_{{w^*}}} < \infty } \right\},
\end{equation}
where
\begin{equation}
{\left\| v \right\|_{{w^*}}} = {\left( {\int\limits_0^\infty  {{{\left| {v\left( x \right)} \right|}^2}{w^*}\left( x \right)dx} } \right)^{\frac{1}{2}}}
\end{equation}
is the norm induced by the scalar product
\begin{equation}
{\left\langle {u,v} \right\rangle _{{w^*}}} = \int\limits_0^\infty  {u\left( x \right)v\left( x \right)} {w^*}\left( x \right)dx\label{inner}
\end{equation}
Thus ${\left\{ {U_n^*\left( x \right)} \right\}_{n \geqslant 0}}$ denote a system which are mutually orthogonal
under \eqref{inner}, i.e
\begin{equation}
{\left\langle {U_n^*\left( x \right),U_m^*\left( x \right)} \right\rangle _{{w^*}}} =\frac{\pi }{2}{\delta _{nm}}
\end{equation}
where $\delta _{nm}$ is the Kronecker delta function. This system is complete in $L_{{w^*}}^2\left( I \right)$. For any function $u\in L_{{w^*}}^2\left( I \right)$ the following expansion holds
\begin{equation}
f(x) \cong \sum\limits_{i = -N}^{+N}  {{a_i}U_i^*(x)}\label{faprox}
\end{equation}
\begin{equation}
{a_k} = \frac{{{{\left\langle {u,U_k^*} \right\rangle }_{{w^*}}}}}{{\left\| {U_k^*} \right\|_{{w^*}}^2}}
\end{equation}
The $a_k$'s are the expansion coefficients associated with the family $\left\{ {U_n^*\left( x \right)} \right\}$.
Now we can define an orthogonal projection based on the transformed Chebyshev functions in second kind as given below: Let
\begin{equation}
{\mho _N^*} = span\{U_0^*(x),U_1^*(x),\ldots,U_n^*(x)\}
                                            \nonumber \\
\end{equation}
The ${L^2}(I)$-orthogonal projection ${\tilde \xi _N}:{L^2}\left( I \right) \to {\mho_N^*}$ is a mapping in
a way that for any $y\in {L^2}(I)$,
\begin{equation}
\left\langle {{{\tilde \xi }_N}y - y,\phi } \right\rangle  = 0,\forall \phi  \in {\mho _N^*}
                                            \nonumber \\
\end{equation}

or equivalently,
\begin{equation}
{{\tilde \xi }_N} y = \sum\limits_{i = 0}^N {{a_i}{U_i^*(x)}}\label{opr}
\end{equation}

\section{Applications}
In this section we have applied Rational Chebyshev of Second kind Collocation method for the computation of Lane-Emden type equations based
on the Rational Chebyshev of Second kind functions. In general the Lane-Emden type equations are formulated as
\begin{equation}
y'' + \frac{\alpha}{x}y' + f(x)g(x) = h(x),~~~\alpha x \geq 0. \label{glane}
\end{equation}
with initial conditions
\begin{equation}
~y(0)=A,~y'(0)=B \nonumber
\end{equation}
where $\alpha$, $A$ and $B$ are real constants and $f(x)$, $g(y)$ and $h(x)$
are some given functions.

We apply the Rational Chebyshev of Second kind functions collocation
method to solve some well-known Lane-Emden type equations for
various $f(x)$, $g(y)$, $A$ and $B$, in two cases homogeneous ($h(x)=0$)
and non-homogeneous ($h(x) \neq 0$).

To satisfy  boundary conditions we can add
the \eqref{opr} by $P(x) = A + Bx$, therefore the \eqref{opr} has been changed as follows:\\
\begin{equation}
{{\tilde \xi }_N} y(x)=P(x)+x^2{{\tilde \xi }_N} y(x).\nonumber
\end{equation}
Now for boundary conditions we have ${{\tilde \xi }_N} y(x)=A$ and
$\frac{d}{{dx}}{{\tilde \xi }_N} y(x)=B$ when $x$ tends to zero.\\

To apply the collocation method, we have constructed the residual
function by substituting $y(x)$ by ${{\tilde \xi }_N} y(x)$ in the Lane-Emden type equation \eqref{glane}:\\
\begin{equation}
\operatorname{Res} (x) = \frac{{{d^2}}}{{d{x^2}}}{{\tilde \xi }_N}y(x) + \frac{\alpha }{x}\frac{d}{{dx}}{{\tilde \xi }_N}y(x) + f(x)g({{\tilde \xi }_N}y(x)) - h(x). \nonumber\\
\end{equation}\\
There is no limitation to choose the point in collocation method. In this case, some points with equal distance from each other in $[0,q]$ have been chosen.~The amount of $q$ has been provided by physical view of Lane-Emden equation.
$N+1$ points with equal distance from each other in $[0,q]$ have been substituted in $Res(x)$:\\
\begin{equation}
Res(x_j) = 0 ,j=0,1,2,3,\ldots,N \nonumber
\end{equation}
A non-linear system of equations has been obtained. Some numerical methods should be used as Newton's method. The Maple version 17 has been used to solve the non-linear system of equations. Solving these types of equations system, Maple software has been used the Newton's method and advanced algorithm. The Maple $fsolve$ has been used for all of unknown coefficients with an initial zero value.
The unknown coefficients $a_i$ will be obtained and approximation of ${{\tilde \xi }_N} y(x)$ is the result.

\subsection{Example 1 (The standard Lane-Emden equation)}
According to Eq.\eqref{glane}, if $f(x) = 1,g(y) = y^m,A=1$ and $B=0$ the standard Lane-Emden equation has been defined as follows:\\
\begin{equation}
y''(x) + \frac{2}{x}y'(x) + {y^m}(x) = 0,x \geqslant 0 \label{sle}
\end{equation}
with boundary conditions\\
\begin{equation}
~y(0)=1,~y'(0)=0 \nonumber
\end{equation}
where $m \geq 0$ is constant.\\
In Eq. \eqref{sle} for $m=0$ exact solution is $y(x) = 1 - \frac{1}{{3!}}{x^2}$, for $m=1$ is $y(x) = \frac{{\sin (x)}}{x}$ and for $m=5$ is $y(x) = {\left( {1 + \frac{{{x^2}}}{3}} \right)^{ - \frac{1}{2}}}$.
\\
The Rational Chebyshev of Second kind Collocation method has been applied to solve Eq. \eqref{sle} for $m=1.5,2,2.5,3$ and $4$ because there is no analytic exact solution for other case of $m$ values so the residual function has been constructed as:
\begin{equation}
\operatorname{Res} (x) = \frac{{{d^2}}}{{d{x^2}}}{{\tilde \xi }_N}y(x) + \frac{2}{x}\frac{d}{{dx}}{{\tilde \xi }_N}y(x) + {\left( {{{\tilde \xi }_N}y(x)} \right)^m}.
\end{equation}
As it was mentioned, $N+1$ points with equal distance from each other in $[0,q]$ have been substituted in $Res(x)$:\\
\begin{equation}
Res(x_j) = 0 ,j=0,1,2,3,\ldots,N \nonumber
\end{equation}
A non-linear system of equations has been obtained. Some numerical methods should be used as Newton's method. The Maple version 17 has been used to solve the non-linear system of equations. Solving these types of equations system, Maple software has been used the Newton's method and advanced algorithm. The Maple $fsolve$ has been used for all of unknown coefficients with an initial zero value.
The unknown coefficients $a_i$ will be obtained and approximation of ${{\tilde \xi }_N} y(x)$ is the result.

Table \ref{table:m table} have shown comparison of the first zeros of standard Lane-Emden equations, for the current
method and exact numerical values given by Horedt \cite{-38} for $m=1.5,2,2.5,3$ and $4$ with $N=20$
\cref{table:2m,table:3m} have represented approximations of $y(x)$ for the standard
Lane-Emden equation for $m=2,3$ respectively obtained by Horedt  \cite{-38} and the method proposed in this paper\\ and it shows the accuracy from this method and the exact value is equal up to 8 digits.

\cref{table:1.5m,table:2.5m,table:4m} have represented approximations of $y(x)$ for the standard
Lane-Emden equation for $m=1.5,2.5,3.5$ and $4$ respectively obtained by by Horedt  \cite{-38} and the method proposed in this paper\\

\cref{table:ai} has shown the coefficients of the Chebyshev of the Second kinds functions obtained by the curent method for $m=1.5,2,2.5,3$ and $4$ of the standard Lane-Emden equation.\\
The result graph of the standard Lane-Emden equation for $m=1.5,2,2.5,3,3.5$ and $4$ has  shown in \cref{fig:mi}.
The logarithmic graph of absolute coefficients of Chebyshev of the Second kinds of functions of standard Lane-Emden equation for $m=1.5$ is shown in \cref{fig:ai}. The coefficients in \cref{table:ai} has shown that there is an appropriate convergence rate for the new method.

\begin{table}[!hbp]
\caption{Comparison of the first zeros of standard Lane-Emden equations, for the current
method and exact numerical values given by Horedt \cite{-38}}

\begin{tabular}{|l|c|c|c|}
\hline
$m$ & N & Present method & Exact Value \\
\hline
$1.50$ & $20$ & $3.65375374E+00$ & $3.65375374E+00$ \\
$2.00$ & $20$ &	$4.35287460E+00$ & $4.35287460E+00$ \\
$2.50$ & $20$ &	$5.35527546E+00$ & $5.35527546E+00$ \\
$3.00$ & $20$ &	$6.89684862E+00$ & $6.89684862E+00$ \\
$4.00$ & $20$ &	$1.49715463E+01$ & $1.49715463E+01$ \\

\hline
\end{tabular}
\label{table:m table}
\end{table}

\begin{table}[!hbp]\label{1.5m}
\caption{Comparison of $y(x)$ values of standard Lane-Emden equation, for the current method and exact values given by Horedt for m=1.5 \cite{-38}}

\begin{tabular}{|l|c|c|c|}
\hline
$x$ & Present method & Exact Value & Error\\
\hline
$0.00$ & $1.00000000E+00$ & $1.00000000E+00$ & $0.00000000E+00$ \\
$0.10$ & $9.98334583E-01$ & $9.98334600E-01$ & $1.70959980E-08$ \\
$0.50$ & $9.59103857E-01$ & $9.59103900E-01$ & $4.30635190E-08$ \\
$1.00$ & $8.45169755E-01$ & $8.45169800E-01$ & $4.45139320E-08$ \\
$3.00$ & $1.58850614E-01$ & $1.58857600E-01$ & $6.98614498E-06$ \\
$3.60$ & $1.10779151E-02$ & $1.10909900E-02$ & $1.30748847E-05$ \\
$3.65$ & $7.62752205E-04$ & $7.63924200E-04$ & $1.17199471E-06$ \\
$3.6537537$	& $2.44970506E-10$ & $0.00000000E+00$ & $0.00000000E+00$ \\
\hline

\end{tabular}
\label{table:1.5m}
\end{table}

\begin{table}[!hbp]
\caption{Comparison of $y(x)$ values of standard Lane-Emden equation, for the current method and exact values given by Horedt for m=2 \cite{-38}}

\begin{tabular}{|l|c|c|c|}
\hline
$x$ & Present method & Exact Value & Error\\
\hline
$0.00$ & $  1.00000000E+00$ & $	1.00000000E+00$ & $	0.00000000E+00$ \\
$0.10$ & $	9.98334988E-01$ & $	9.98335000E-01$ & $	1.17588950E-08$ \\
$0.50$ & $	9.59352705E-01$ & $	9.59352700E-01$ & $	4.98044506E-09$ \\
$1.00$ & $	8.48654103E-01$ & $	8.48654100E-01$ & $	2.86298907E-09$ \\
$3.00$ & $	2.41824078E-01$ & $	2.41824100E-01$ & $	2.16966440E-08$ \\
$4.00$ & $	4.88401323E-02$ & $	4.88401500E-02$ & $	1.77484112E-08$ \\
$4.30$ & $	6.81092861E-03$ & $	6.81094300E-03$ & $	1.43878525E-08$ \\
$4.35$ & $	3.66029676E-04$ & $	3.66030200E-04$ & $	5.24343265E-10$ \\
$4.35287460$ & $	0.00000000E+00$ & $	0.00000000E+00$ & $	0.00000000E+00$ \\

\hline

\end{tabular}
\label{table:2m}
\end{table}

\begin{table}[!hbp]
\caption{Comparison of $y(x)$ values of standard Lane-Emden equation, for the current method and exact values given by Horedt for m=2.5 \cite{-38}}

\begin{tabular}{|l|c|c|c|}
\hline
$x$ & Present method & Exact Value & Error\\
\hline
$0.00$ & $	1.00000000E+00$ & $	1.00000000E+00$ & $	0.00000000E+00$ \\
$0.10$ & $	9.98200161E-01$ & $	9.98335400E-01$ & $	1.35238951E-04$ \\
$0.50$ & $	9.59597754E-01$ & $	9.59597800E-01$ & $	4.55336809E-08$ \\
$1.00$ & $	8.51944199E-01$ & $	8.51944200E-01$ & $	8.71904993E-10$ \\
$4.00$ & $	1.37680751E-01$ & $	1.37680700E-01$ & $	5.09351070E-08$ \\
$5.00$ & $	2.90198639E-02$ & $	2.90191900E-02$ & $	6.73893582E-07$ \\
$5.30$ & $	4.25986419E-03$ & $	4.25954400E-03$ & $	3.20185644E-07$ \\
$5.355$ & $	2.10110537E-05$ & $	2.10089000E-05$ & $	2.15367148E-09$ \\
$5.35527546$ & $	0.00000000E+00$ & $	0.00000000E+00$ & $	0.00000000E+00$ \\

\hline

\end{tabular}
\label{table:2.5m}
\end{table}

\begin{table}[!hbp]
\caption{Comparison of $y(x)$ values of standard Lane-Emden equation, for the current method and exact values given by Horedt for m=3 \cite{-38}}

\begin{tabular}{|l|c|c|c|}
\hline
$x$ & Present method & Exact Value & Error\\
\hline
$0.00 $ & $ 1.00000000E+00$ & $	1.00000000E+00$ & $	0.00000000E+00$ \\
$0.10 $ & $ 9.98335820E-01$ & $	9.98335800E-01$ & $	2.00000000E-08$ \\
$0.50 $ & $ 9.59839060E-01$ & $	9.59839100E-01$ & $	4.00000000E-08$ \\
$1.00 $ & $ 8.55057560E-01$ & $	8.55057600E-01$ & $	4.00000000E-08$ \\
$5.00 $ & $ 1.10819830E-01$ & $	1.10819800E-01$ & $	3.00000000E-08$ \\
$6.00 $ & $ 4.37379800E-02$ & $	4.37380000E-02$ & $	2.00000000E-08$ \\
$6.80 $ & $ 4.16778000E-03$ & $	4.16780000E-03$ & $	2.00000000E-08$ \\
$6.90 $ & $-1.33650000E-04$ & $	3.60000000E-05$ & $	1.69650000E-04$ \\

\hline

\end{tabular}
\label{table:3m}
\end{table}

\begin{table}[!hbp]
\caption{Comparison of $y(x)$ values of standard Lane-Emden equation, for the current method and exact values given by Horedt for m=4 \cite{-38}}

\begin{tabular}{|l|c|c|c|}
\hline
$x$ & Present method & Exact Value & Error\\
\hline
$0.00 $ & $1.00000000E+00$ & $	1.00000000E+00$ & $	0.00000000E+00$ \\
$0.10 $ & $1.01585088E+00$ & $	9.98336700E-01$ & $	1.75141815E-02$ \\
$0.20 $ & $9.92425032E-01$ & $	9.93386200E-01$ & $	9.61167905E-04$ \\
$0.50 $ & $9.60310891E-01$ & $	9.60310900E-01$ & $	8.81922002E-09$ \\
$1.00 $ & $8.60813812E-01$ & $	8.60813800E-01$ & $	1.22079751E-08$ \\
$5.00 $ & $2.35922731E-01$ & $	2.35922700E-01$ & $	3.10685290E-08$ \\
$10.00 $ & $5.96737717E-02$ & $	5.96727400E-02$ & $	1.03172317E-06$ \\
$14.00 $ & $8.33293035E-03$ & $	8.33052700E-03$ & $	2.40334699E-06$ \\
$14.90 $ & $5.76673836E-04$ & $	5.76418900E-04$ & $	2.54935916E-07$ \\

\hline

\end{tabular}
\label{table:4m}
\end{table}

\begin{table}[!hbp]
\caption{Coefficients of the RCS functions of the standard Lane-Emden equations for $m=1.5,2,3$ and $4$}
\begin{tabular}{|l|lllll|}
\hline
$i$ & $a_i$    &     &       &     &     \\ \cline{2-6}
  & $m=1.5$ & $m=2$ & $m=2.5$ & $m=3$ & $m=4$ \\ \hline
0 &  -8.35020235E-02  &  -7.37260675E-02  &  -5.84134154E-02  &  -6.00541969E-01  &  -1.80533469E-02 \\
 1&  4.14892772E-02  &  4.61551717E-02  &  4.24877591E-02  &  1.07145455E+00  &  1.60801870E-03 \\
 2&  1.06527872E-02  &  -5.59588313E-04  &  9.68579002E-03  &  -1.38004834E+00 & 4.91151095E-02 \\
 3& -6.46059281E-03 & -2.88264289E-04 & -5.69688651E-03 & 1.57447823E+00 & -6.29427069E-02 \\
 4& -9.23628705E-03 & -2.03783735E-02 & -2.83354601E-03 & -1.61608168E+00 & 5.51425972E-02 \\
 5& -5.14378683E-03 & 2.77045133E-03 & -1.88446376E-03 & 1.49866216E+00 & -5.60724094E-02 \\
 6& -3.99197152E-04 & -1.30872723E-02 & 2.47503817E-03 & -1.28488320E+00 & 5.00091630E-02 \\
 7& 2.35669147E-03 & 6.02895464E-03 & 4.20009874E-04 & 1.01947325E+00 & -4.03638627E-02 \\
 8& 3.03971126E-03 & -8.47557449E-03 & 1.86252286E-03 & -7.46940934E-01 & 3.09682458E-02 \\
 9& 2.53242065E-03 & 4.00096435E-03 & -1.11760997E-04 & 5.05881970E-01 & -2.19137068E-02 \\
 10& 1.67607539E-03 & -4.94521907E-03 & 7.92650268E-04 & -3.15031208E-01 & 1.42286699E-02 \\
 11& 9.36473928E-04 & 1.94713984E-03 & -1.16018023E-04 & 1.79413013E-01 & -8.51801989E-03 \\
 12& 4.51996582E-04 & -2.06474187E-03 & 3.37002942E-04 & -9.27123693E-02 & 4.63985922E-03 \\
 13& 1.89757483E-04 & 7.66054296E-04 & -1.36758208E-05 & 4.30569575E-02 & -2.27214677E-03 \\
 14& 6.90455024E-05 & -5.75592272E-04 & 1.09800992E-04 & -1.76863574E-02 & 9.90326051E-04 \\
 15& 2.15040362E-05 & 2.13683933E-04 & 4.50594242E-06 & 6.32155250E-03 & -3.76662896E-04 \\
 16& 5.58921881E-06 & -9.82199992E-05 & 2.08231737E-05 & -1.90188385E-03 & 1.20837515E-04 \\
 17& 1.16162729E-06 & 3.53103253E-05 & 1.59326696E-06 & 4.63035504E-04 & -3.14970993E-05 \\
 18& 1.75835253E-07 & -7.82964769E-06 & 1.71120723E-06 & -8.26242326E-05 & 6.09225231E-06 \\
 19& 1.59423494E-08 & 2.50723737E-06 & 1.81759246E-07 & 9.25997591E-06 & -7.06745003E-07 \\ \hline

\end{tabular}
\label{table:ai}
\end{table}

\begin{figure}[!hbp]
\centerline{\includegraphics[totalheight=6cm]{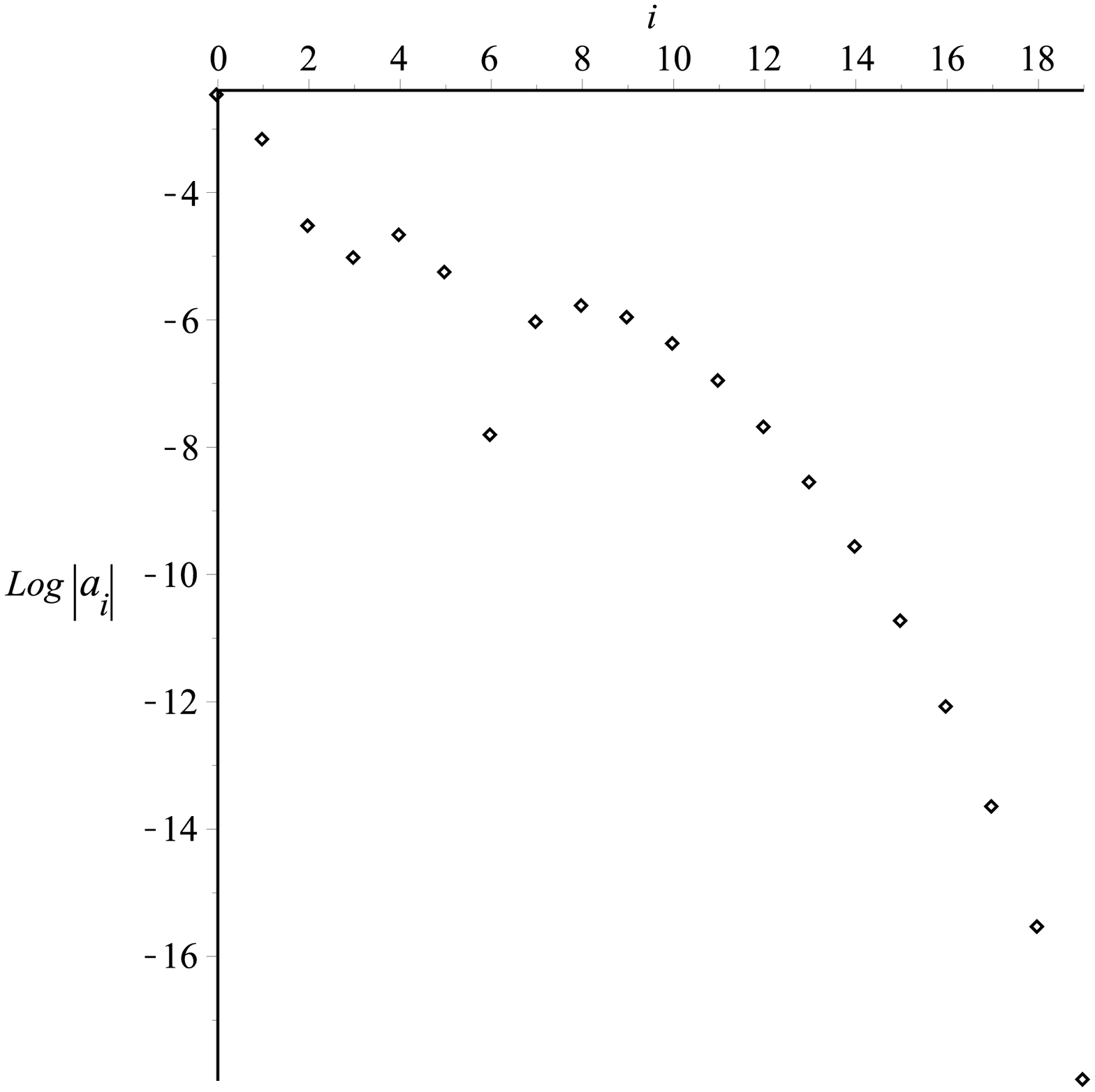}}
 \caption{Logarithmic graph of absolute coefficients $|a_i|$ of RCS function of standard
Lane-Emden equation form=1.5}

    \label{fig:ai}
\end{figure}

\begin{figure}[!htp]
\centerline{\includegraphics[scale=0.4,natwidth=610,natheight=642]{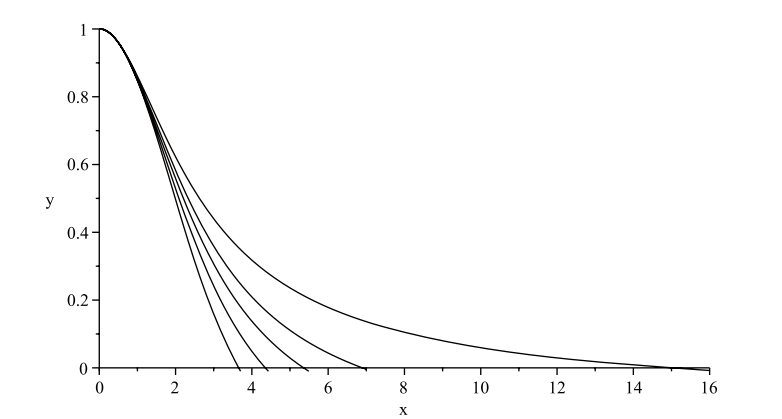}}
 \caption{Graph of standard Lane-Emden equation for $m=1.5, 2, 2.5, 3$ and $4$.}

    \label{fig:mi}
\end{figure}
\newpage
\subsection{Example 2 (The isothermal gas spheres equation)}
According to Eq.\eqref{glane}, if $f(x) = 1,g(y) = e^y,A=0$ and $B=0$ the isothermal gas spheres equation has been defined as follows:\\
\begin{equation}
y''(x) + \frac{2}{x}y'(x) + {e^y(x)} = 0,x \geqslant 0 \label{eq:isothermal}
\end{equation}
with boundary conditions\\
\begin{equation}
y(0)=0,y'(0)=0 \nonumber
\end{equation}
Davis \cite{davis 2006} has discussed about \cref{eq:isothermal} that can be used to view the isothermal gas spheres, where
the temperature remains constant.\\
A series solution have investigated by Wazwaz \cite{29},Liao \cite{33},Singh et al. \cite{49} and Ramos \cite{37} by using ADM, ADM, MHAM and series expansion, respectively:
\begin{equation}
y(x) \simeq  - \frac{1}{6}{x^2} + \frac{1}{{5.4!}}{x^4} - \frac{8}{{21.6!}}{x^6} + \frac{{122}}{{81.8!}}{x^8} - \frac{{61.67}}{{459.10!}}{x^{10}}.
\end{equation}
We have applied Chebyshev of Second kind Collocation method to solve Eq. \eqref{eq:isothermal} therefore we have constructed the residual function:
\begin{equation}
\operatorname{Res} (x) = \frac{{{d^2}}}{{d{x^2}}}{{\tilde \xi }_N}y(x) + \frac{2}{x}\frac{d}{{dx}}{{\tilde \xi }_N}y(x) + { e^{{{\tilde \xi }_N}y(x)}}.
\end{equation}
As it was mentioned, $N+1$ points with equal distance from each other in $[0,q]$ have been substituted in $Res(x)$:\\
\begin{equation}
Res(x_j) = 0 ,j=0,1,2,3,\ldots,N \nonumber
\end{equation}
A non-linear system of equations has been obtained. Some numerical methods should be used as Newton's method. The Maple version 17 has been used to solve the non-linear system of equations. Solving these types of equations system, Maple software has been used the Newton's method and advanced algorithm. The Maple $fsolve$ has been used for all of unknown coefficients with an initial zero value.
The unknown coefficients $a_i$ will be obtained and approximation of ${{\tilde \xi }_N} y(x)$ is the result.

\cref{table:isothermal} has shown the comparison of $y(x)$ obtained by the method
proposed in this paper with $N=40$ and those obtained by Wazwaz \cite{29}. As ADM is based on Taylor series ,it can be contributed that there is a very high level of accuracy around zero point. This accuracy will decrease while the distance of points starts to raise from zero point.
The resulting graph of the isothermal gas spheres equation
in comparison to the presented method and those obtained by
Wazwaz\cite{29} has been shown in \cref{fig:isothermal:waz}.
The logarithmic graph of the absolute coefficients of Chebyshev of Second kind functions of the standard isothermal
gas spheres is shown in \cref{fig:isothermal:ai}. This graph shows that the new
method has a proper convergence rate.

\begin{table}[!hbp]
\caption{Comparison of $y(x)$, between current method and series solution given by Wazwaz
\cite{29} for isothermal gas sphere equation.}

\begin{tabular}{|l|c|c|c|}
\hline
$x$ & Present method & Wazwaz & Error\\
\hline
$0.00 $ & $  0.00000000E+00 $ & $ 0.00000000E+00 $ & $ 0.00000000E+00 $\\
$0.10 $ & $ -1.66583386E-03 $ & $ -1.66583390E-03 $ & $ 3.79394198E-11 $\\
$0.20 $ & $ -6.65336710E-03 $ & $ -6.65336710E-03 $ & $ 4.23139822E-13 $\\
$0.50 $ & $ -4.11539573E-02 $ & $ -4.11539568E-02 $ & $ 4.92713800E-10 $\\
$1.00 $ & $ -1.58827678E-01 $ & $ -1.58827354E-01 $ & $ 3.23824394E-07 $\\
$1.50 $ & $ -3.38019425E-01 $ & $ -3.38013110E-01 $ & $ 6.31446080E-06 $\\
$2.00 $ & $ -5.59823004E-01 $ & $ -5.59962660E-01 $ & $ 1.39655764E-04 $\\
$2.50 $ & $ -8.06340871E-01 $ & $ -8.10019671E-01 $ & $ 3.67880071E-03 $\\

\hline

\end{tabular}
\label{table:isothermal}
\end{table}

\begin{figure}[!hbp]
\centerline{\includegraphics[totalheight=6cm]{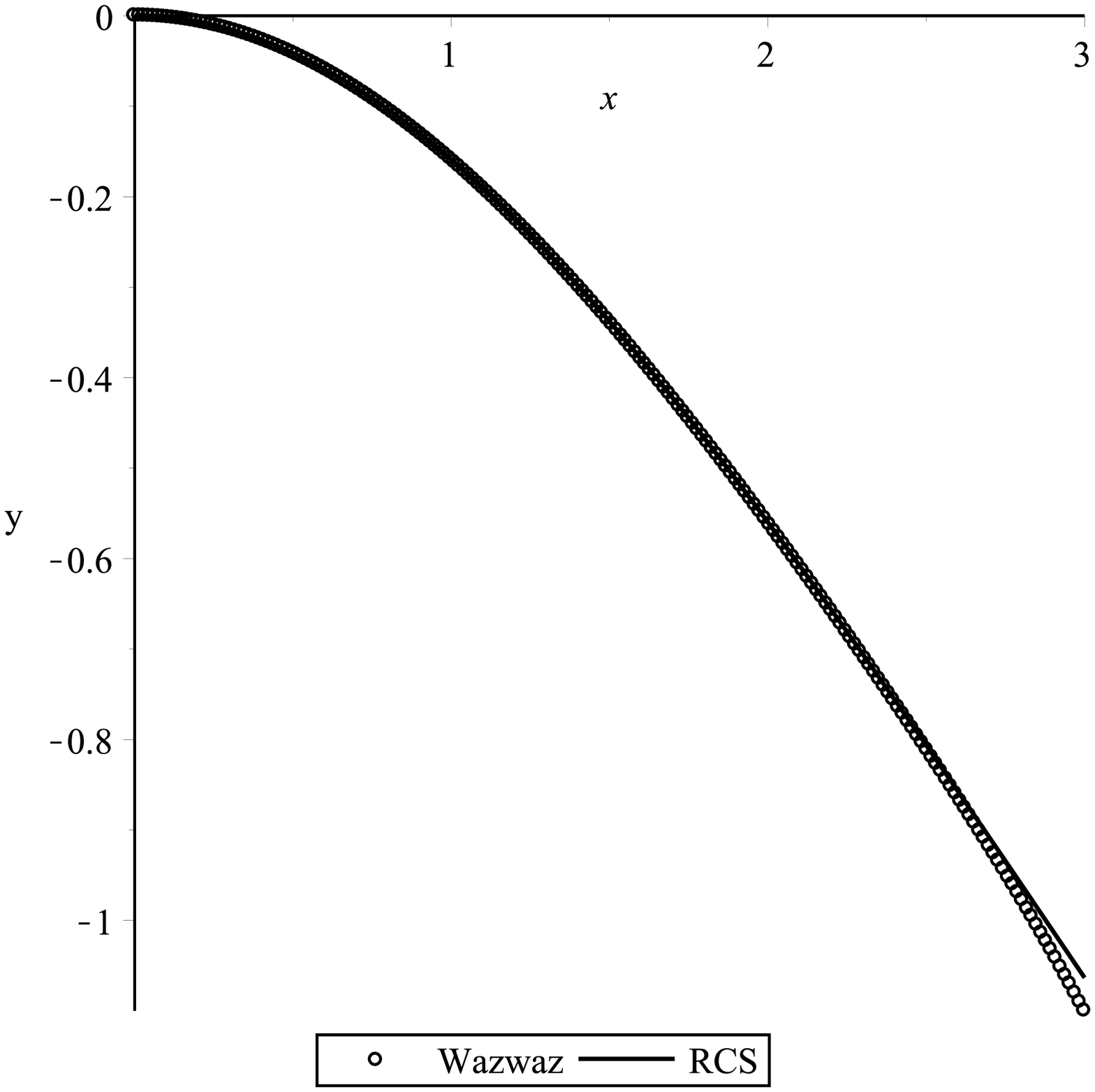}}
 \caption{Graph of isothermal gas sphere equation in comparison with Wazwaz solution \cite{29}}

    \label{fig:isothermal:waz}
\end{figure}

\begin{figure}
\centerline{\includegraphics[totalheight=6cm]{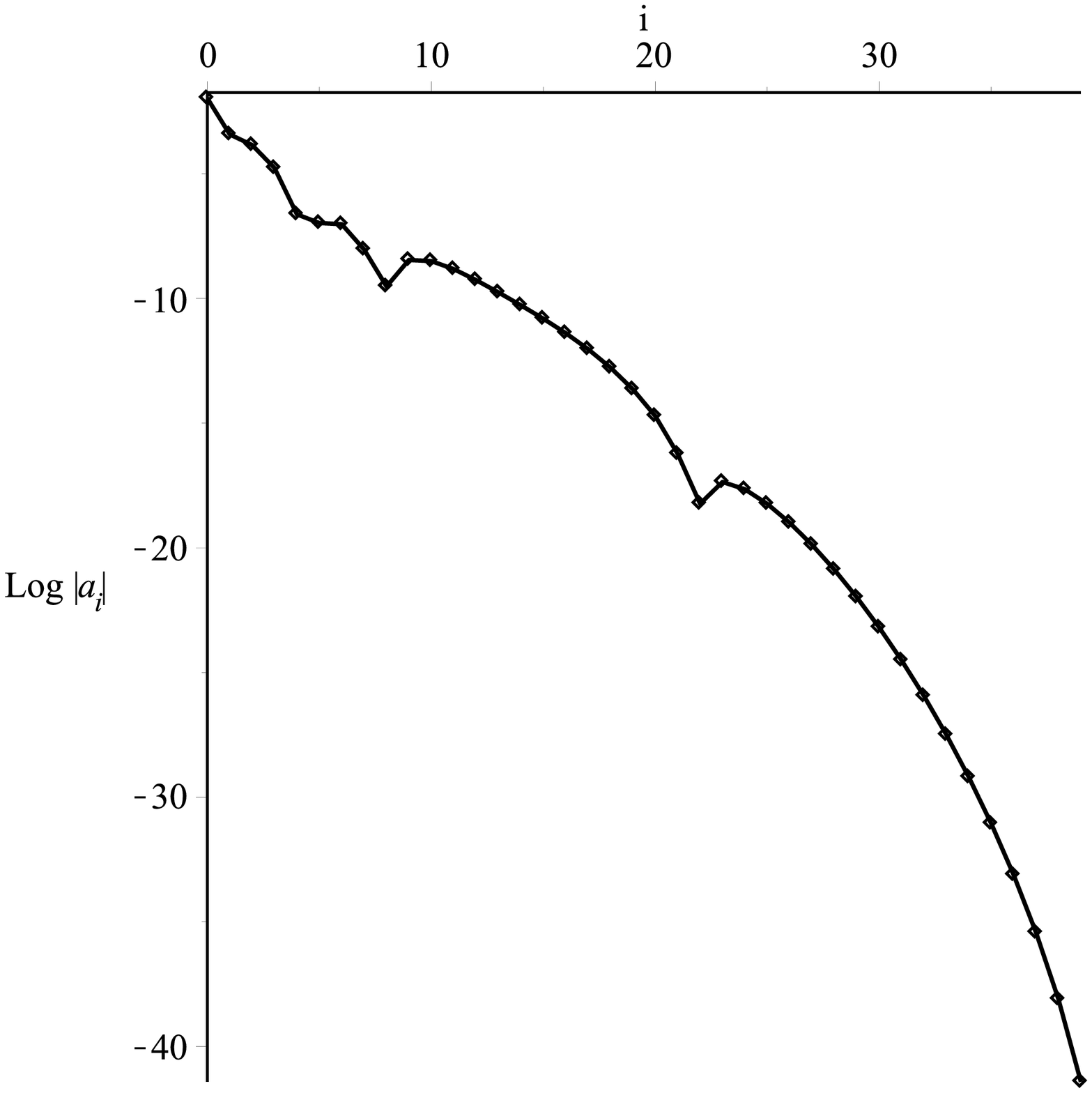}}
 \caption{Logarithmic graph of absolute coefficients $|ai|$ of RCS function of isothermal gas sphere equation.}

    \label{fig:isothermal:ai}
\end{figure}
\newpage
\subsection{Example 3}
According to Eq.\eqref{glane}, if $f(x) = 1,g(y) = \sinh (y),A=1$ and $B=0$ the equation has been defined as follows:\\
\begin{equation}
y''(x) + \frac{2}{x}y'(x) + {\sinh (y)} = 0,x \geqslant 0 \label{eq:sinh}
\end{equation}
with boundary conditions
\begin{equation}
y(0)=1,y'(0)=0\nonumber
\end{equation}

A series solution have investigated by Wazwaz \cite{29} by using Adomian Decomposition Method (ADM) is:
\begin{equation}
y(x) \simeq 1 - \frac{{({e^2} - 1){x^2}}}{{12e}} + \frac{1}{{480}}\frac{{({e^4} - 1){x^4}}}{{{e^2}}} - \frac{1}{{30240}}\frac{{(2{e^6} + 3{e^2} - 3{e^4} - 2){x^6}}}{{{e^3}}} + \frac{1}{{26127360}}\frac{{(61{e^8} + 104{e^6} - 104{e^2} - 61){x^8}}}{{{e^4}}}
\end{equation}
we have purposed to RCS method to solve \cref{eq:sinh} therefore, we construct the residual function as follows:
\begin{equation}
\operatorname{Res} (x) = \frac{{{d^2}}}{{d{x^2}}}{{\tilde \xi }_N}y(x) + \frac{2}{x}\frac{d}{{dx}}{{\tilde \xi }_N}y(x) + {\sinh ({{{\tilde \xi }_N}y(x)})}.
\end{equation}\\

$N+1$ points with equal distance from each other in $[0,q]$ have been substituted in $Res(x)$:\\
\begin{equation}
Res(x_j) = 0 ,j=0,1,2,3,\ldots,N \nonumber
\end{equation}
A non-linear system of equations has been obtained. Some numerical methods should be used as Newton's method. The Maple version 17 has been used to solve the non-linear system of equations. Solving these types of equations system, Maple software has been used the Newton's method and advanced algorithm. The Maple $fsolve$ has been used for all of unknown coefficients with an initial zero value.
The unknown coefficients $a_i$ will be obtained and approximation of ${{\tilde \xi }_N} y(x)$ is the result.

\cref{table:exmp3} has shown the comparison of $y(x)$ obtained by the new
method proposed in this paper with N=20, and
those obtained by Wazwaz \cite{29}.The resulting graph of \cref{eq:sinh}
in comparison to the presented method and those obtained by
Wazwaz \cite{29} are shown in \cref{fig:exmp3:waz}.

\begin{table}[!htp]
\caption{Comparison of $y(x)$, between current method and series solution given by Wazwaz
\cite{29} for Example 3.}

\begin{tabular}{|l|c|c|c|}
\hline
$x$ & Present method & Wazwaz & Error\\
\hline
$0.00$ & $0.00000000E+00$ & $0.00000000E+00$ & $0.00000000E+00$\\
$0.10$ & $9.98042842E-01$ & $9.98042841E-01$ & $9.80743042E-10$\\
$0.20$ & $9.92189436E-01$ & $9.92189435E-01$ & $1.02811104E-09$\\
$0.50$ & $9.51961094E-01$ & $9.51961102E-01$ & $8.14090706E-09$\\
$1.00$ & $8.18242929E-01$ & $8.18251667E-01$ & $8.73755808E-06$\\
$1.50$ & $6.25438765E-01$ & $6.25891608E-01$ & $4.52843155E-04$\\
$2.00$ & $4.06623301E-01$ & $4.13669104E-01$ & $7.04580312E-03$\\

\hline

\end{tabular}
\label{table:exmp3}
\end{table}

\begin{figure}[!htp]
\centerline{\includegraphics[totalheight=6cm]{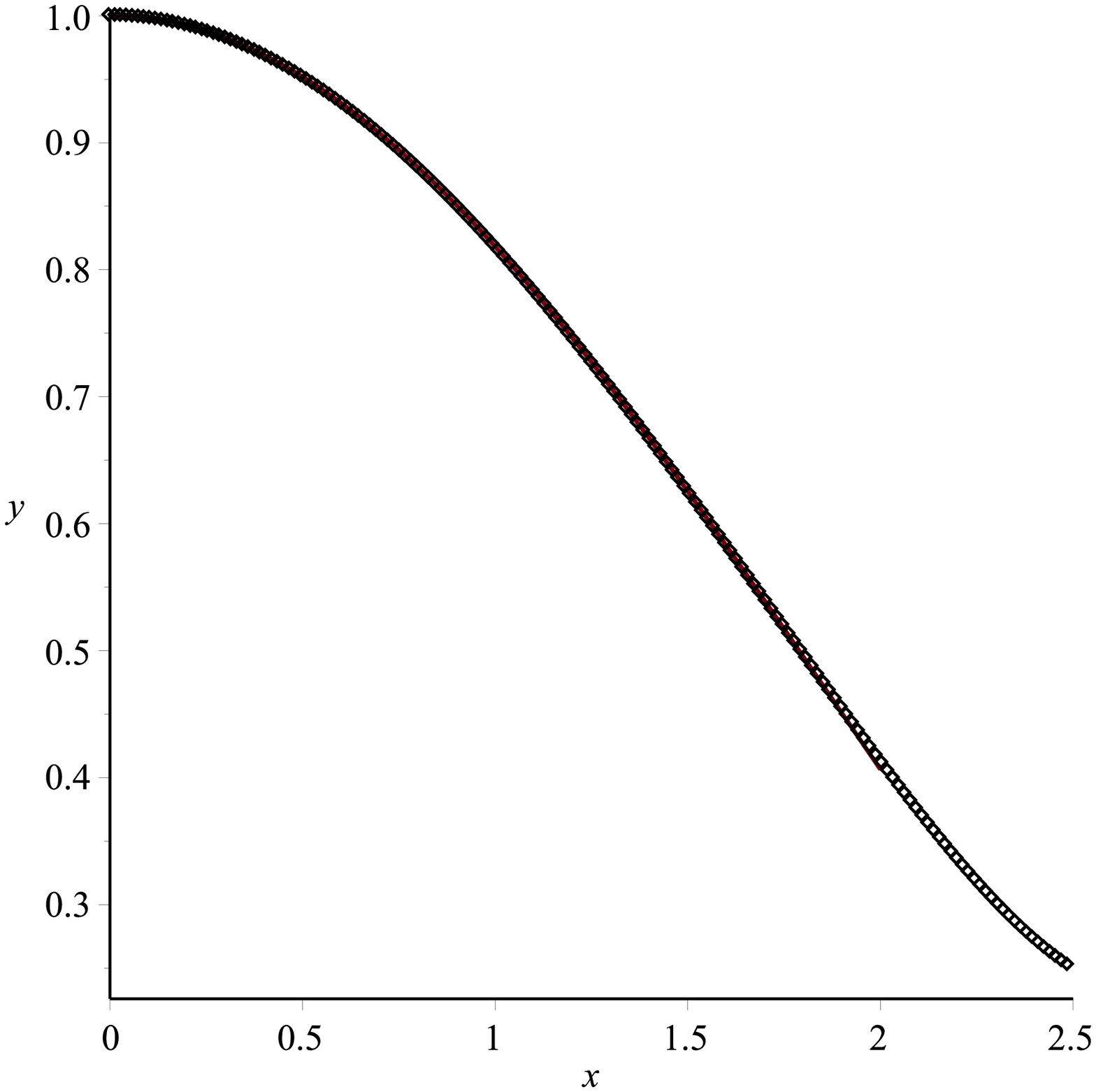}}
 \caption{Graph of equation of Example 3 in comparing the presented method and Wazwaz solution \cite{29}}

    \label{fig:exmp3:waz}
\end{figure}
\newpage
\subsection{Example 4}
According to Eq.\eqref{glane}, if $f(x) = 1,g(y) = \sin (y),A=1$ and $B=0$ the equation has been defined as follows:\\
\begin{equation}
y''(x) + \frac{2}{x}y'(x) + {\sin (y)} = 0,x \geqslant 0 \label{eq:sin}
\end{equation}
with boundary conditions
\begin{equation}
y(0)=1,y'(0)=0\nonumber
\end{equation}

A series solution have investigated by Wazwaz \cite{29} by using Adomian Decomposition Method (ADM) is:
\begin{align}
y(x) &\simeq 1 - \frac{1}{6}{k_1}{x^2} + \frac{1}{{120}}{k_1}{k_2}{x^4} + {k_1}\left( {\frac{1}{{3024}}k_1^2 - \frac{1}{{5040}}k_2^2} \right){x^6} + {k_1}{k_2}\left( { - \frac{{113}}{{3265920}}k_1^2 + \frac{1}{{362880}}k_2^2} \right){x^8} \nonumber \\
& + {k_1}\left( {\frac{{1781}}{{82892800}}k_1^2k_2^2 - \frac{1}{{399168000}}k_2^4 - \frac{{19}}{{23950080}}k_1^4} \right){x^{10}}
\end{align}
where $k_1=\sin (1)$ and $k_2=\cos (1)$
we have purposed to RCS method to solve \cref{eq:sin} therefore, we construct the residual function as follows:
\begin{equation}
\operatorname{Res} (x) = \frac{{{d^2}}}{{d{x^2}}}{{\tilde \xi }_N}y(x) + \frac{2}{x}\frac{d}{{dx}}{{\tilde \xi }_N}y(x) + {\sin ({{{\tilde \xi }_N}y(x)})}.
\end{equation}\\

As it was mentioned, $N+1$ points with equal distance from each other in $[0,q]$ have been substituted in $Res(x)$:\\
\begin{equation}
Res(x_j) = 0 ,j=0,1,2,3,\ldots,N \nonumber
\end{equation}
A non-linear system of equations has been obtained. Some numerical methods should be used as Newton's method. The Maple version 17 has been used to solve the non-linear system of equations. Solving these types of equations system, Maple software has been used the Newton's method and advanced algorithm. The Maple $fsolve$ has been used for all of unknown coefficients with an initial zero value.
The unknown coefficients $a_i$ will be obtained and approximation of ${{\tilde \xi }_N} y(x)$ is the result.

\cref{table:exmp4} has shown the comparison of $y(x)$ obtained by the new
method proposed in this paper with N=20, and
those obtained by Wazwaz \cite{29}.The resulting graph of \cref{eq:sin}
in comparison to the current method and those obtained by
Wazwaz \cite{29} are shown in \cref{fig:exmp4:waz}.

\begin{table}[!htp]
\caption{Comparison of $y(x)$, between current method and series solution given by Wazwaz
\cite{29} for Example 4.}

\begin{tabular}{|l|c|c|c|}
\hline
$x$ & Present method & Wazwaz & Error\\
\hline
$0.00$ & $1.00000000E+00$ & $1.00000000E+00$ & $0.00000000E+00$\\
$0.10$ & $9.98597930E-01$ & $9.98597936E-01$ & $5.40042000E-09$\\
$0.20$ & $9.94396268E-01$ & $9.94396273E-01$ & $5.12808607E-09$\\
$0.50$ & $9.65177784E-01$ & $9.65177789E-01$ & $5.03923003E-09$\\
$1.00$ & $8.63681129E-01$ & $8.63681103E-01$ & $2.60689941E-08$\\
$1.50$ & $7.05045237E-01$ & $7.05041925E-01$ & $3.31264027E-06$\\
$2.00$ & $5.06464502E-01$ & $5.06372033E-01$ & $9.24689456E-05$\\

\hline

\end{tabular}
\label{table:exmp4}
\end{table}

\begin{figure}[!htp]
\centerline{\includegraphics[totalheight=6cm]{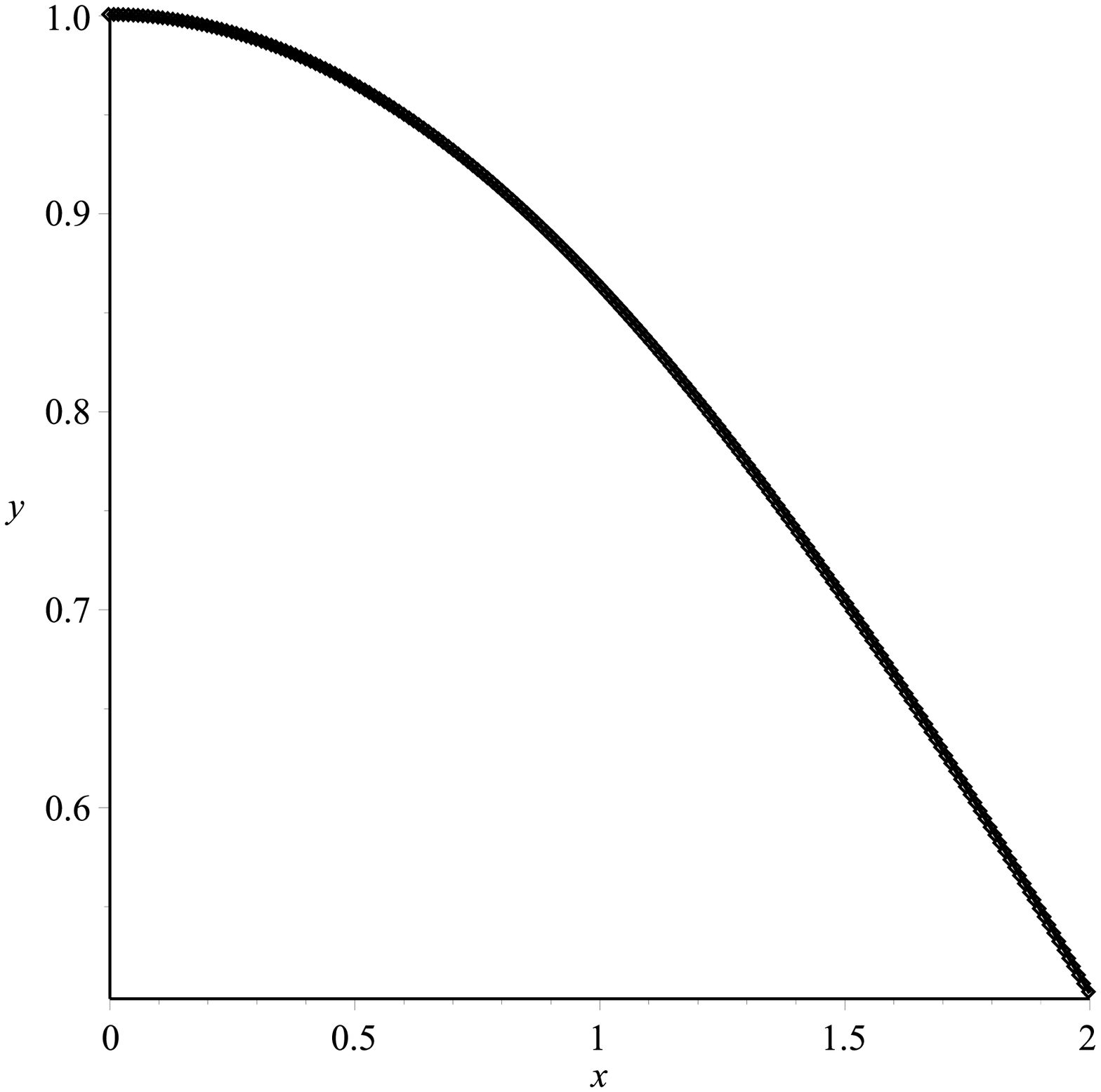}}
 \caption{Graph of equation of Example 4 in comparing the current method and Wazwaz solution \cite{29}}

    \label{fig:exmp4:waz}
\end{figure}
\newpage
\subsubsection{Example 5}
According to Eq.\eqref{glane}, if $f(x) = 1, g(y) = 4(2{e^y} + {e^{\frac{y}{2}}}),A=0$ and $B=0$ the equation has been defined as follows:\\
\begin{equation}
y''(x) + \frac{2}{x}y'(x) + {4(2{e^y} + {e^{\frac{y}{2}}})} = 0,x \geqslant 0 \label{eq:exmp5}
\end{equation}
with boundary conditions
\begin{equation}
y(0)=0,y'(0)=0\nonumber
\end{equation}

which has the following analytical solution:
\begin{equation}
y(x)=-2\ln(1+x^2). \label{eq:exmp5:anal}
\end{equation}

We have purposed to RCS method to solve \cref{eq:exmp5} therefore, we construct the residual function as follows:
\begin{equation}
\operatorname{Res} (x) = \frac{{{d^2}}}{{d{x^2}}}{{\tilde \xi }_N}y(x) + \frac{2}{x}\frac{d}{{dx}}{{\tilde \xi }_N}y(x) +
 {4(2{e^{{{\tilde \xi }_N}y(x)}} + {e^{\frac{{{{\tilde \xi }_N}y(x)}}{2}}})}.
\end{equation}\\

$N+1$ points with equal distance from each other in $[0,q]$ have been substituted in $Res(x)$:\\
\begin{equation}
Res(x_j) = 0 ,j=0,1,2,3,\ldots,N \nonumber
\end{equation}
A non-linear system of equations has been obtained. Some numerical methods should be used as Newton's method. The Maple version 17 has been used to solve the non-linear system of equations. Solving these types of equations system, Maple software has been used the Newton's method and advanced algorithm. The Maple $fsolve$ has been used for all of unknown coefficients with an initial zero value.
The unknown coefficients $a_i$ will be obtained and approximation of ${{\tilde \xi }_N} y(x)$ is the result.

\cref{table:exmp5} has shown the comparison of $y(x)$ obtained by the new
method proposed in this paper with N=46, and
the analytic solution\cref{eq:exmp5:anal} .The resulting graph of \cref{eq:exmp5}
in comparison to the presented method and
the analytic solution\cref{eq:exmp5:anal} are shown in \cref{fig:exmp5:waz}.

\begin{table}[!htp]
\caption{Comparison of $y(x)$, between present method and exact solution for Example 5.}

\begin{tabular}{|l|c|c|c|}
\hline
$x$ & Present method & Exact value & Error\\
\hline
$0.00$ & $0.00000000E+00$ & $0.00000000E+00$ & $0.00000000E+00$\\
$0.01$ & $-1.99985444E-04$ & $-1.99990000E-04$ & $4.55642163E-09$\\
$0.10$ & $-1.99006434E-02$ & $-1.99006617E-02$ & $1.83196007E-08$\\
$0.50$ & $-4.46287089E-01$ & $-4.46287103E-01$ & $1.39323760E-08$\\
$1.00$ & $-1.38629436E+00$ & $-1.38629436E+00$ & $4.32735003E-09$\\
$2.00$ & $-3.21887583E+00$ & $-3.21887582E+00$ & $3.97273014E-09$\\
$3.00$ & $-4.60517019E+00$ & $-4.60517019E+00$ & $5.26229016E-09$\\
$4.00$ & $-5.66642669E+00$ & $-5.66642669E+00$ & $4.81089035E-09$\\
$5.00$ & $-6.51619308E+00$ & $-6.51619308E+00$ & $3.89665988E-09$\\
$6.00$ & $-7.22183583E+00$ & $-7.22183583E+00$ & $1.57089008E-09$\\
$7.00$ & $-7.82404602E+00$ & $-7.82404601E+00$ & $6.89516000E-09$\\
$8.00$ & $-8.34877473E+00$ & $-8.34877454E+00$ & $1.88648450E-07$\\
$9.00$ & $-8.81343999E+00$ & $-8.81343849E+00$ & $1.49618078E-06$\\
$10.00$ & $-9.23024811E+00$ & $-9.23024103E+00$ & $7.08015080E-06$\\

\hline

\end{tabular}
\label{table:exmp5}
\end{table}

\begin{figure}[!htp]
\centerline{\includegraphics[totalheight=6cm]{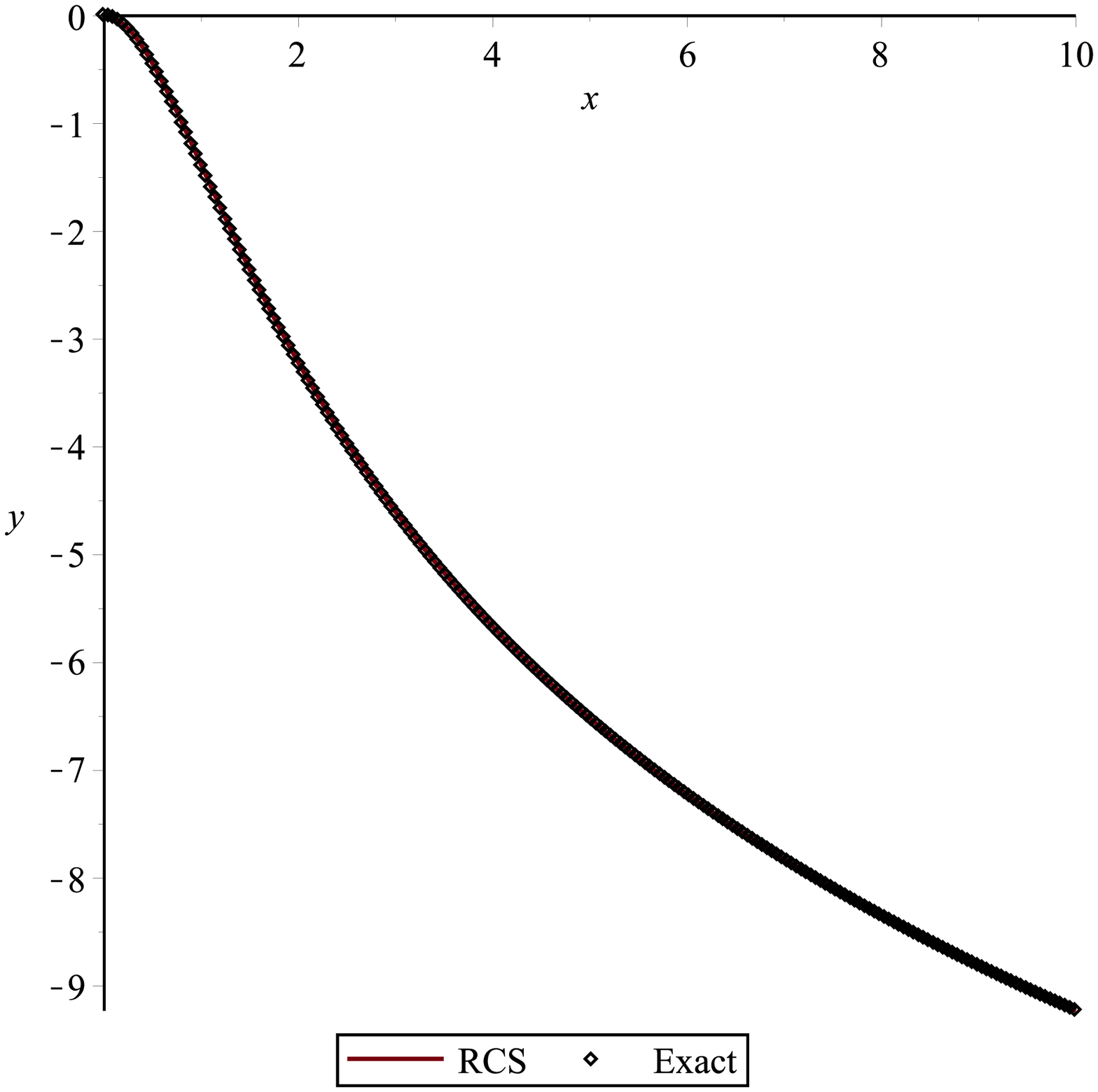}}
 \caption{Graph of equation of Example 5 in comparing the presented method and
analytic solution.}

    \label{fig:exmp5:waz}
\end{figure}
\newpage
\subsubsection{Example 6}
If $f(x) = 1, g(y) = -6y-4y\ln (y),A=0$ and $B=1$, one of the Lane-Emden type equations will obtain as follows:\\
\begin{equation}
y''(x) + \frac{2}{x}y'(x)-6y-4y\ln (y) = 0,x \geqslant 0 \label{eq:exmp6}
\end{equation}
with boundary conditions
\begin{equation}
y(0)=1,y'(0)=0\nonumber
\end{equation}

which has the following analytical solution:
\begin{equation}
y(x)=e^{x^2}. \label{eq:exmp6:anal}
\end{equation}

We have purposed to RCS method to solve \cref{eq:exmp6} therefore, we construct the residual function as follows:
\begin{equation}
\operatorname{Res} (x) = \frac{{{d^2}}}{{d{x^2}}}{{\tilde \xi }_N}y(x) + \frac{2}{x}\frac{d}{{dx}}{{\tilde \xi }_N}y(x) -
 {6{{{\tilde \xi }_N}y(x)}} - {4{{{\tilde \xi }_N}y(x)}\ln ({{{\tilde \xi }_N}y(x)})} .
\end{equation}\\

$N+1$ points with equal distance from each other in $[0,q]$ have been substituted in $Res(x)$:\\
\begin{equation}
Res(x_j) = 0 ,j=0,1,2,3,\ldots,N \nonumber
\end{equation}
A non-linear system of equations has been obtained. Some numerical methods should be used as Newton's method. The Maple version 17 has been used to solve the non-linear system of equations. Solving these types of equations system, Maple software has been used the Newton's method and advanced algorithm. The Maple $fsolve$ has been used for all of unknown coefficients with an initial zero value.
The unknown coefficients $a_i$ will be obtained and approximation of ${{\tilde \xi }_N} y(x)$ is the result.

\cref{table:exmp6} has shown the comparison of $y(x)$ obtained by the new
method proposed in this paper with N=40, and
the analytic solution\cref{eq:exmp6:anal} .The resulting graph of \cref{eq:exmp6}
in comparison to the presented method and
the analytic solution\cref{eq:exmp6:anal} are shown in \cref{fig:exmp6:waz}.

The logarithmic graph of the absolute coefficients of Chebyshev of Second kind functions of \cref{eq:exmp6} is shown in \cref{fig:exmp6:ai}. This graph shows that the new method has a proper convergence rate.

\begin{table}[!htp]
\caption{Comparison of $y(x)$, between present method and exact solution for Example 6.}

\begin{tabular}{|l|c|c|c|}
\hline
$x$ & Present method & Exact value & Error\\
\hline
$0.00$ & $1.00000000E+00$ & $1.00000000E+00$ & $0.00000000E+00$\\
$0.01$ & $1.00010001E+00$ & $1.00010001E+00$ & $5.02500264E-11$\\
$0.02$ & $1.00040008E+00$ & $1.00040008E+00$ & $1.58560054E-10$\\
$0.05$ & $1.00250313E+00$ & $1.00250313E+00$ & $2.54499977E-10$\\
$0.10$ & $1.01005017E+00$ & $1.01005017E+00$ & $2.65950151E-10$\\
$0.20$ & $1.04081077E+00$ & $1.04081077E+00$ & $3.03500114E-10$\\
$0.50$ & $1.28402542E+00$ & $1.28402542E+00$ & $4.30919966E-10$\\
$0.70$ & $1.63231622E+00$ & $1.63231622E+00$ & $5.87929927E-10$\\
$0.80$ & $1.89648088E+00$ & $1.89648088E+00$ & $7.86489984E-10$\\
$0.90$ & $2.24790799E+00$ & $2.24790799E+00$ & $9.62370184E-10$\\
$1.00$ & $2.71828183E+00$ & $2.71828183E+00$ & $1.22766020E-09$\\

\hline

\end{tabular}
\label{table:exmp6}
\end{table}

\begin{figure}[!htp]
\centerline{\includegraphics[totalheight=6cm]{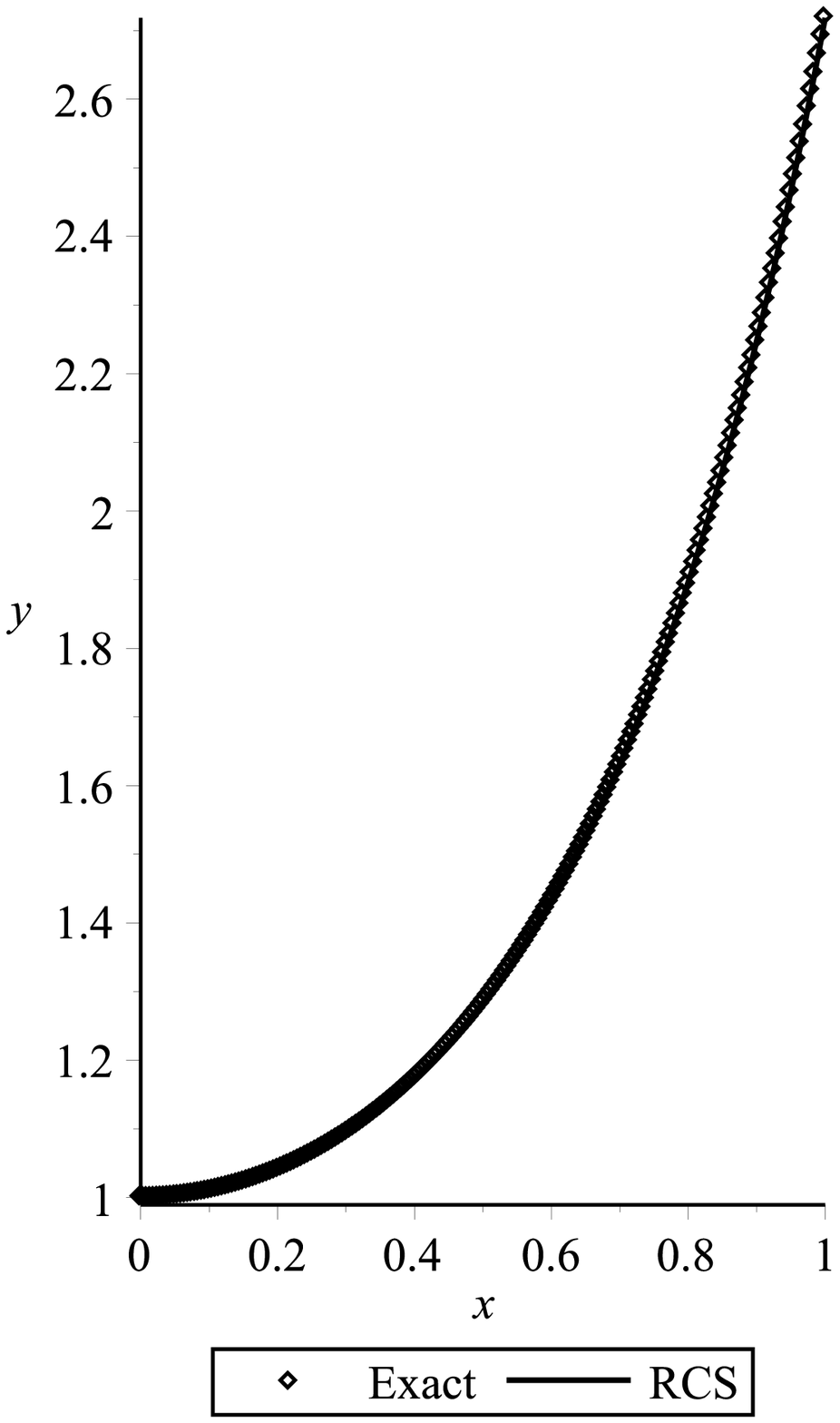}}
 \caption{Graph of equation of Example 6 in comparing the current method and
analytic solution.}

    \label{fig:exmp6:waz}
\end{figure}

\begin{figure}[!htp]
\centerline{\includegraphics[totalheight=6cm]{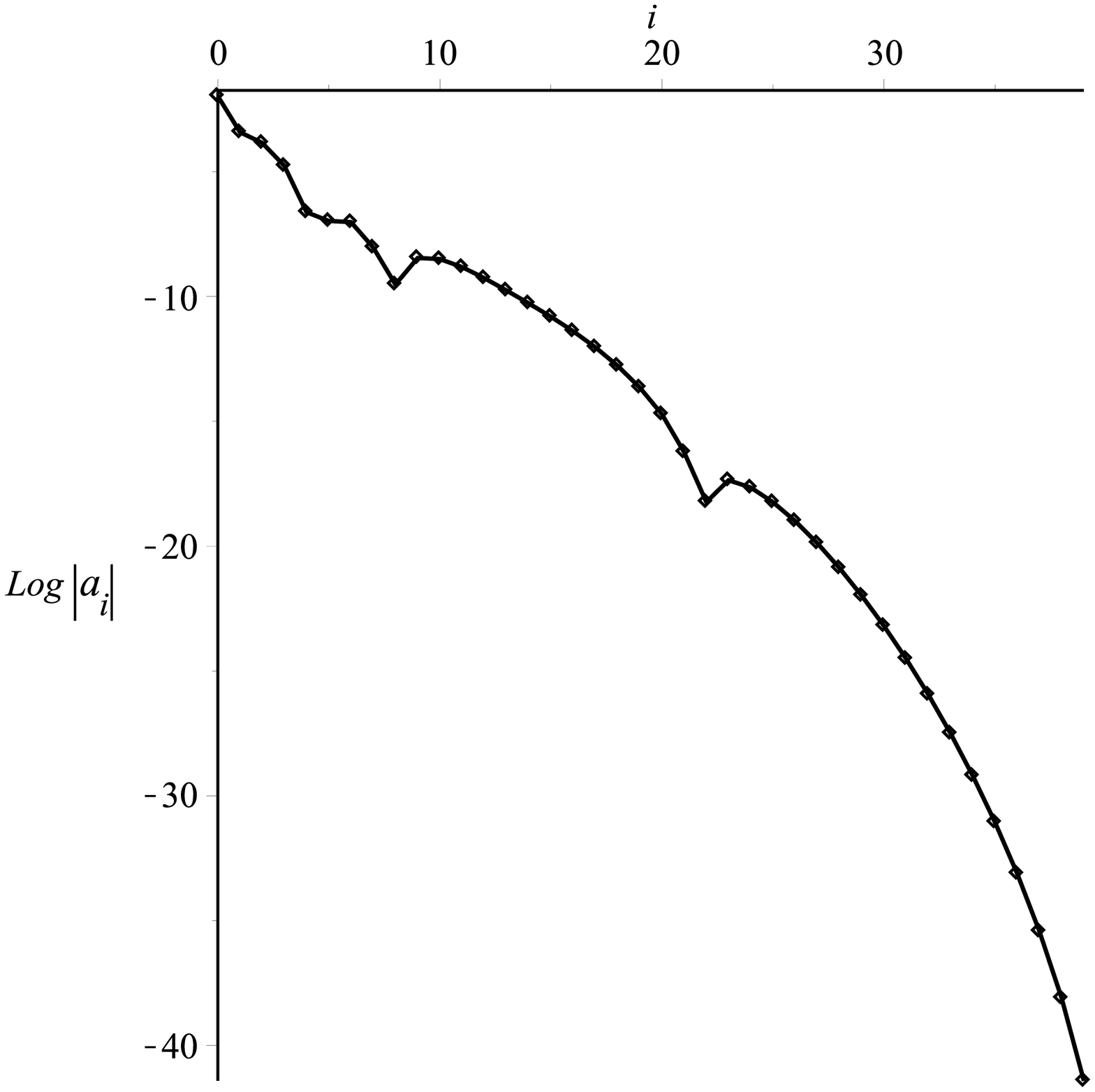}}
 \caption{Logarithmic graph of absolute coefficients $|ai|$ of RCS function of Example 6}

    \label{fig:exmp6:ai}
\end{figure}
\newpage
\subsubsection{Example 7}
If $f(x) = -2(2x^2+3), g(y) = y,A=1$ and $B=0$, one of the Lane-Emden type equations will obtain as follows:\\
\begin{equation}
y''(x) + \frac{2}{x}y'(x)-2(2x^2+3)y = 0,x \geqslant 0 \label{eq:exmp7}
\end{equation}
with boundary conditions
\begin{equation}
y(0)=1,y'(0)=0\nonumber
\end{equation}

which has the following analytical solution:
\begin{equation}
y(x)=e^{x^2}. \label{eq:exmp7:anal}
\end{equation}

We have purposed to RCS method to solve \cref{eq:exmp6} therefore, we construct the residual function as follows:
\begin{equation}
\operatorname{Res} (x) = \frac{{{d^2}}}{{d{x^2}}}{{\tilde \xi }_N}y(x) + \frac{2}{x}\frac{d}{{dx}}{{\tilde \xi }_N}y(x) -
 2(2x^2+3){{{\tilde \xi }_N}y(x)}.
\end{equation}\\

$N+1$ points with equal distance from each other in $[0,q]$ have been substituted in $Res(x)$:\\
\begin{equation}
Res(x_j) = 0 ,j=0,1,2,3,\ldots,N \nonumber
\end{equation}
A non-linear system of equations has been obtained. Some numerical methods should be used as Newton's method. The Maple version 17 has been used to solve the non-linear system of equations. Solving these types of equations system, Maple software has been used the Newton's method and advanced algorithm. The Maple $fsolve$ has been used for all of unknown coefficients with an initial zero value.
The unknown coefficients $a_i$ will be obtained and approximation of ${{\tilde \xi }_N} y(x)$ is the result.

\cref{table:exmp7} has shown the comparison of $y(x)$ obtained by the new
method proposed in this paper with N=40, and
the analytic solution\cref{eq:exmp7:anal} .The resulting graph of \cref{eq:exmp7}
in comparison to the presented method and
the analytic solution\cref{eq:exmp7:anal} are shown in \cref{fig:exmp7:waz}.

The logarithmic graph of the absolute coefficients of Rational Chebyshev of Second kind functions of \cref{eq:exmp7} is shown in \cref{fig:exmp7:ai}. This graph shows that the new method has a proper convergence rate.

\begin{table}[!htp]
\caption{Comparison of $y(x)$, between current method and exact solution for Example 7.}

\begin{tabular}{|l|c|c|c|}
\hline
$x$ & Present method & Exact value & Error\\
\hline
$0.00$ & $1.00000000E+00$ & $1.00000000E+00$ & $0.00000000E+00$\\
$0.01$ & $1.00010001E+00$ & $1.00010001E+00$ & $4.58300065E-11$\\
$0.02$ & $1.00040008E+00$ & $1.00040008E+00$ & $1.07270193E-10$\\
$0.05$ & $1.00250313E+00$ & $1.00250313E+00$ & $1.38300038E-10$\\
$0.10$ & $1.01005017E+00$ & $1.01005017E+00$ & $1.27110100E-10$\\
$0.20$ & $1.04081077E+00$ & $1.04081077E+00$ & $1.44440016E-10$\\
$0.50$ & $1.28402542E+00$ & $1.28402542E+00$ & $1.78340009E-10$\\
$0.70$ & $1.63231622E+00$ & $1.63231622E+00$ & $1.98119965E-10$\\
$0.80$ & $1.89648088E+00$ & $1.89648088E+00$ & $2.87079915E-10$\\
$0.90$ & $2.24790799E+00$ & $2.24790799E+00$ & $3.10949932E-10$\\
$1.00$ & $2.71828183E+00$ & $2.71828183E+00$ & $3.63560293E-10$\\

\hline

\end{tabular}
\label{table:exmp7}
\end{table}

\begin{figure}[!htp]
\centerline{\includegraphics[totalheight=6cm]{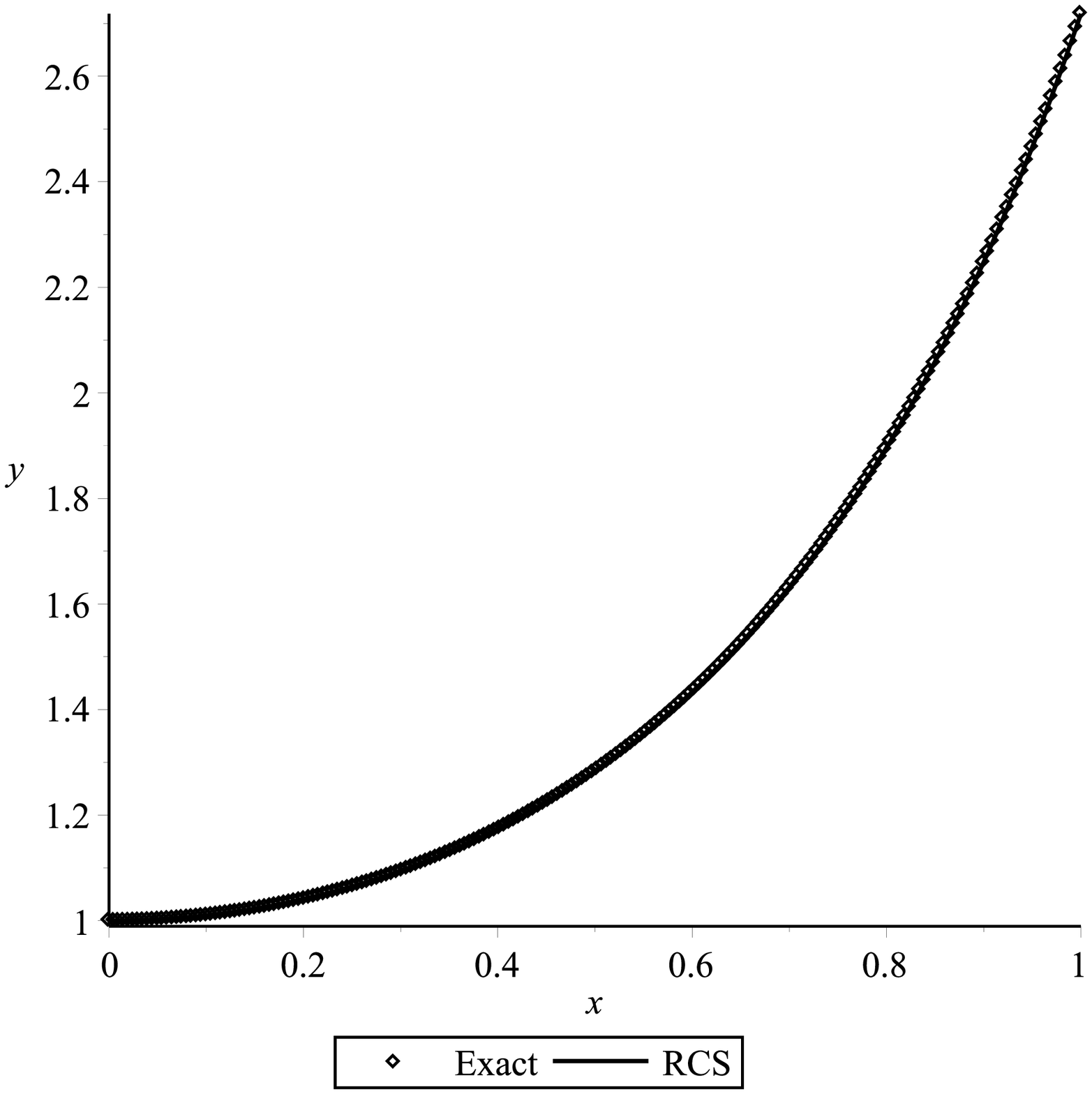}}
 \caption{Graph of equation of Example 7 in comparing the presented method and
analytic solution.}

    \label{fig:exmp7:waz}
\end{figure}

\begin{figure}[!htp]
\centerline{\includegraphics[totalheight=6cm]{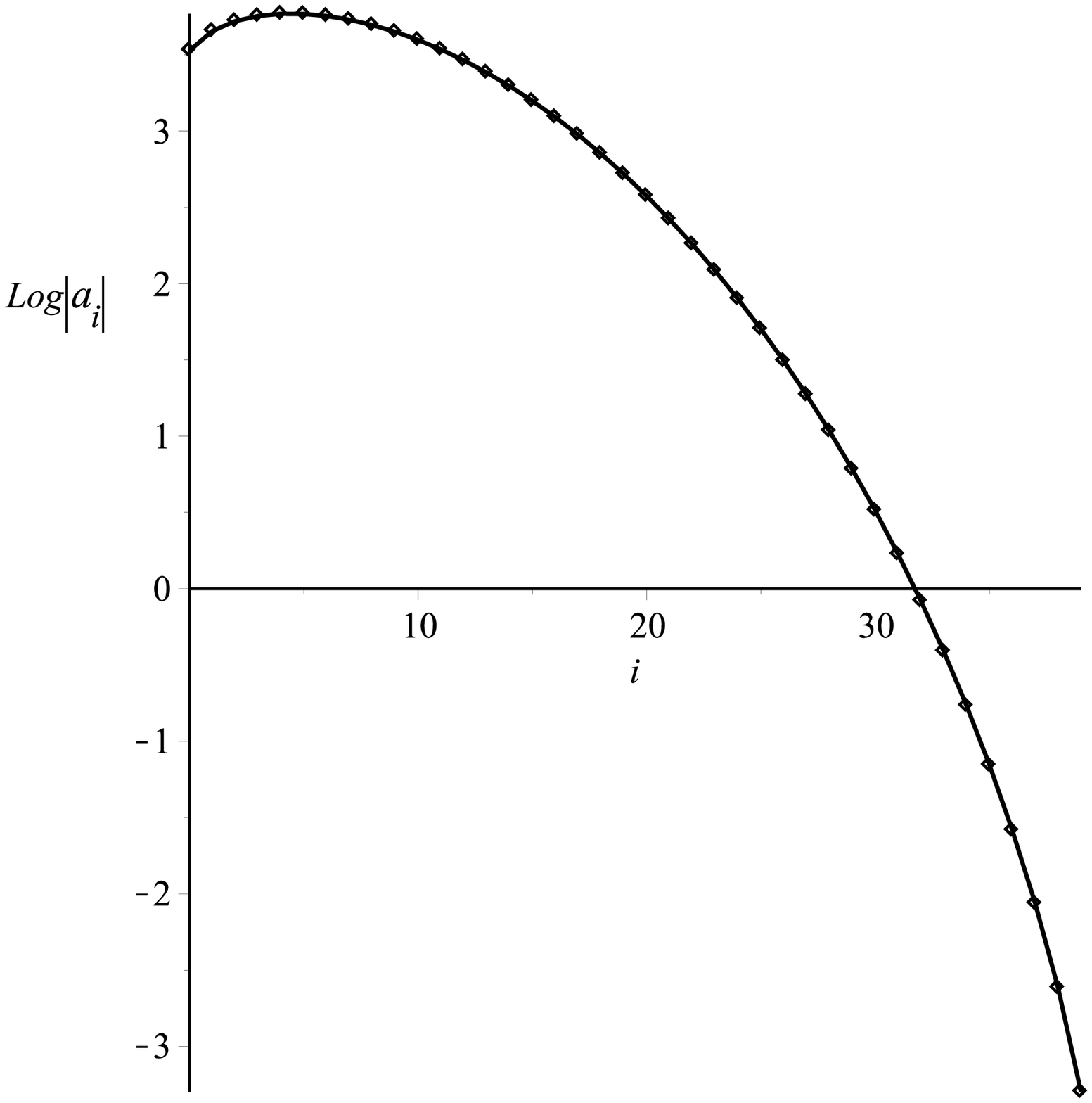}}
 \caption{Logarithmic graph of absolute coefficients $|ai|$ of RCS function of Example 7}

    \label{fig:exmp7:ai}
\end{figure}
\newpage
\section{Conclusions}
The main idea of this paper is to construct an approximation to the solution of nonlinear Lane-Emden equations in a semi-infinite interval. A set of
Rational Chebyshev of the Second kinds of functions have been proposed to provide an effective but simple way to improve the convergence of the solution by the collocation method. there is a comparison among the exact solution and the numerical solutions of Horedt \cite{-38} and the series solutions of Wazwaz
\cite{29},Liao\cite{33}, Singh et al. \cite{49} and Ramos\cite{37} and the current work. As a result, the present work provides an acceptable approach for Lane-Emden equations. The logarithmic figures of absolute coefficients result in the approach which has exponentially convergence rate. Finally, the most important concern of spectral methods is the choice of basis functions. there are some properties for basis functions such as rapid convergence, easy to compute and completeness, which means that any solution can be represented to arbitrarily high accuracy by taking the truncation N
are sufficiently large.


\clearpage
\begin{thebibliography}{99}
\bibitem{could 90}O. Coulaud, D. Funaro, O. Kavian, Laguerre spectral approximation of elliptic problems in exterior domains, Comput. Method. Appl. M 80 (1990) 451-458.
\bibitem{funaro 91}D. Funaro, O. Kavian, Approximation of some diffusion evolution equations in
unbounded domains by Hermite functions, Math. Comput. 57 (1991) 597-619.
\bibitem{funaro 90} D. Funaro, Computational aspects of pseudospectral Laguerre approximations,
Appl. Numer. Math. 6 (1990) 447-457.
\bibitem{gua 99} B.Y. Guo, Error estimation of Hermite spectral method for nonlinear partial differential equations, Math. Comput. 68 (227) (1999) 1067-1078.
\bibitem{gua 2000a} B.Y. Guo, J. Shen, Laguerre-Galerkin method for nonlinear partial differential
equations on a semi-infinite interval, Numer. Math. 86 (2000) 635-654.
\bibitem{mady 85} Y. Maday, B. Pernaud-Thomas, H. Vandeven, Reappraisal of Laguerre type spectral methods, La Rech. Aerospatiale 6 (1985) 13-35.
\bibitem{shen 2000} J. Shen, Stable and efficient spectral methods in unbounded domains using Laguerre functions, SIAM J. Numer. Anal. 38 (2000) 1113-1133.
\bibitem{siyyam 2001} H.I. Siyyam, Laguerre Tau methods for solving higher order ordinary differential
equations, J. Comput. Anal. Appl. 3 (2001) 173-182.
\bibitem{gua 98} B.Y. Guo, Gegenbauer approximation and its applications to differential equations on the whole line, J. Math. Anal. Appl. 226 (1998) 180-206.
\bibitem{gua 2000b} B.Y. Guo, Jacobi spectral approximation and its applications to differential equations on the half line, J. Comput. Math. 18 (2000) 95-112.
\bibitem{gua 2000c} B.Y. Guo, Jacobi approximations in certain Hilbert spaces and their applications
to singular differential equations, J. Math. Anal. Appl. 243 (2000) 373-408.
\bibitem{boyd 2000} J.P. Boyd, Chebyshev and Fourier Spectral Methods, second edition, Dover, New
York, 2000.
\bibitem{christov 1982} C.I. Christov, A complete orthogonal system of functions in $L^{2}(-\infty, +\infty)$ space,SIAM J. Appl. Math. 42 (1982) 1337-1344.
\bibitem{boyd 1987a} J.P. Boyd, Spectral methods using rational basis functions on an infinite interval,J. Comput. Phys. 69 (1987) 112-142.
\bibitem{boyd 1987b} J.P. Boyd, Orthogonal rational functions on a semi-infinite interval, J. Comput.Phys. 70 (1987) 63-88.
\bibitem{gushen 2000}B.Y. Guo, J. Shen, Z.Q. Wang, A rational approximation and its applications to
differential equations on the half line, J. Sci. Comput. 15 (2000) 117-147.

\bibitem{paraz 2004a}K. Parand, M. Razzaghi, Rational Chebyshev tau method for solving Volterra's population model, Appl. Math. Comput. 149 (2004) 893-900.
\bibitem{paraz 2004b}K. Parand, M. Razzaghi, Rational Chebyshev tau method for solving higher-order ordinary differential equations, Int. J. Comput. Math. 81 (2004) 73-80.
\bibitem{paraz 2004c}K. Parand, M. Razzaghi, Rational Legendre approximation for solving some physical problems on semi-infinite intervals, Phys. Scr. 69 (2004) 353-357.
\bibitem{parnd shahini}K. Parand, M. Shahini, Rational Chebyshev pseudospectral approach for solving Thomas-Fermi equation, Phys. Lett. A 373 (2009) 210-213.
\bibitem{paraz taghavi}K. Parand, A. Taghavi, Rational scaled generalized Laguerre function collocation method for solving the Blasius equation, J. Comput. Applied. Math. 233 (2009) 980-989
\bibitem{bender 1986}C.M. Bender, K.A. Milton, S.S. Pinsky, L.M. Simmons Jr., A new perturbative approach to nonlinear problems, J. Math. Phys. 30 (1989) 1447-1455.
\bibitem{mandel 2001}V.B. Mandelzweig, F. Tabakin, Quasilinearization approach to nonlinear problems in physics with application to nonlinear ODEs, Comput. Phys. Commun. 141 (2001) 268-281.
\bibitem{shawag 1993}N.T. Shawagfeh, Nonperturbative approximate solution for Lane-Emden equation, J. Math. Phys. 34 (1993) 4364-4369.
\bibitem{ramos 2005}J.I. Ramos, Linearization techniques for singular initial-value problems of ordinary differential equations, Appl. Math. Comput. 161 (2005) 525-542.
\bibitem{yousefi 2006}S.A. Yousefi, Legendre wavelets method for solving differential equations of
Lane-Emden type, Appl. Math. Comput. 181 (2006) 1417-1422.
\bibitem{davis 2006}Davis, H.T., Introduction to Nonlinear Differential and Integral
Equations. Dover publication, Inc., New York, 1962.
\bibitem{chandr 1967}S. Chandrasekhar, Introduction to the Study of Stellar Structure, Dover, New
York, 1967.
\bibitem{dehghan 2008}M. Dehghan, F. Shakeri, Approximate solution of a differential equation arising
in astrophysics using the variational iteration method, New Astron. 13 (2008)53-59.
\bibitem{agarwal 2007} R.P. Agarwal, D. O'Regan, Second order initial value problems of Lane-Emden
type, Appl. Math. Lett. 20 (2007) 1198-1205.
\bibitem{bender89} C.M. Bender, K.A. Milton, S.S. Pinsky, L.M. Simmons Jr., A new perturbative approach to nonlinear problems, J. Math. Phys. 30 (1989) 1447-1455.
\bibitem{kara92} A.H. Kara, F.M. Mahomed, Equivalent Lagrangians and solutions of some classes of nonlinear equations, Int. J. Nonlinear Mech. 27 (1992) 919-9
\bibitem{kara93} A.H. Kara, F.M. Mahomed, A note on the solutions of the Emden-Fowler equation, Int. J. Nonlinear Mech. 28 (1993) 379-384.
\bibitem{29} A. Wazwaz, A new algorithm for solving differential equations of Lane-Emden type, Appl. Math. Comput. 118 (2001) 287-310.
\bibitem{33} S. Liao, A new analytic algorithm of Lane-Emden type equations, Appl. Math. Comput. 142 (2003) 1-16.
\bibitem{34} J.H. He, Variational approach to the Lane-Emden equation, Appl. Math. Comput. 143 (2003) 539-541.
\bibitem{35} J.I. Ramos, Linearization methods in classical and quantum mechanics, Comput.Phys. Commun. 153 (2003) 199-208.
\bibitem{20} K. Parand, M. Razzaghi, Rational Legendre approximation for solving some physical problems on semi-infinite intervals, Phys. Scripta 69 (2004) 353-357
\bibitem{36} J.I. Ramos, Linearization techniques for singular initial-value problems of ordinary differential equations, Appl. Math. Comput. 161 (2005) 525-542.
\bibitem{32} A. Wazwaz, The modified decomposition method for analytic treatment of differential equations, Appl. Math. Comput. 173 (2006) 165-176.
\bibitem{37} J.I. Ramos, Series approach to the Lane-Emden equation and comparison with the homotopy perturbation method, Chaos Solitons Fractals 38 (2008) 400-408.
\bibitem{46} A. Aslanov, Determination of convergence intervals of the series solutions of Emden-Fowler equations using polytropes and isothermal spheres, Phys. Lett. A 372 (2008) 3555-3561.
\bibitem{48} H.R. Marzban, H.R. Tabrizidooz, M. Razzaghi, Hybrid functions for nonlinear initial-value problems with applications to Lane-Emden type equations, Phys. Lett. A 372 (37) (2008) 5883-5886.
\bibitem{55} M. Dehghan, F. Shakeri, Approximate solution of a differential equation arising in astrophysics using the variational iteration method, New Astron. 13 (2008)
\bibitem{42} M.S.H. Chowdhury, I. Hashim, Solutions of Emden-Fowler equations by homotopy perturbation method, Nonlinear Anal. Real World Appl. 10 (2009) 104-115.
\bibitem{23} K. Parand, M. Shahini, M. Dehghan, Rational Legendre pseudospectral approach for solving nonlinear differential equations of Lane-Emden type, J. Comput.Phys. 228 		(2009) 8830-8840.
\bibitem{38} J.I. Ramos, Piecewise-adaptive decomposition methods, Chaos Solitons Fractals 40 (2009) 1623-1636.
\bibitem{40} A. S Bataineh, M.S.M. Noorani, I. Hashim, Homotopy analysis method for singular IVPs of Emden-Fowler type, Commun. Nonlinear Sci. Numer. Simul. 14 (2009) 1121-		1131.
\bibitem{47} A. Yildirim, T. Özi ¸s, Solutions of singular IVPs of Lane-Emden type by the variational iteration method, Nonlinear Anal. Ser. A Theor. Methods Appl. 70 		(2009)2480-2484.
\bibitem{49} O.P. Singh, R.K. Pandey, V.K. Singh, An analytic algorithm of Lane-Emden type equations arising in astrophysics using modified homotopy analysis 		 method,Comput.Phys. Commun. 180 (2009) 1116-1124.

\bibitem{-1} K. Parand, A. Taghavi, A. Shahini, Comparison between rational Chebyshev and modified generalized Laguerre functions pseudospectral methods for solving Lane-Emden and unsteady gas equations, Acta Physica Polonica B. 40 (2009) 1749-1763.
\bibitem{-2} O.P. Singh, R.K. Pandey, V.K. Singh, An analytic algorithm of Lane-Emden type equations arising in astrophysics using modified Homotopy analysis method
Comput. Phys. Commu. 180 (2009) 1116-1124.
\bibitem{-3} K. Parand, M. Shahini, M. Dehghan, Rational Legendre pseudospectral approach for solving nonlinear differential equations of Lane-Emden type , Journal of Comput. Phys. 228 (2009) 8830-8840.
\bibitem{-4} F. Geng, M. Cui, B. Zhang, Method for solving nonlinear initial value problems by combining homotopy perturbation and reproducing kernel Hilbert space methods
 Nonlinear Analysis: Real World App.  11 (2010) 637-644.
\bibitem{-5} S. Karimi Vanani,A. Aminataei, On the numerical solution of differential equations of Lane-Emden type, Comput. and Math. with App. 59 (2010) 2815-2820.
\bibitem{-6}  van Gorder, R.A.
Analytical solutions to a quasilinear differential equation related to the Lane-Emden equation of the second kind
(2011) Celestial Mechanics and Dynamical Astronomy, 109 (2), pp. 137-145.
\bibitem{-7}  Boyd, J.P.
Chebyshev spectral methods and the Lane-Emden problem
(2011) Numerical Mathematics, 4 (2), pp. 142-157.
\bibitem{-8}  Wazwaz, A.-M., Rach, R.
Comparison of the Adomian decomposition method and the variational iteration method for solving the Lane-Emden equations of the first and second kinds
(2011) Kybernetes, 40 (9), pp. 1305-1318.
\bibitem{-9}  Yüzbasi, S.
A numerical approach for solving a class of the nonlinear Lane-Emden type equations arising in astrophysics
(2011) Mathematical Methods in the Applied Sciences, 34 (18), pp. 2218-2230.
\bibitem{-10}  Bhrawy, A.H., Rezaei, A., Boubaker, K.
Jacobi Pseudo-Spectral Method JPSM and BPES for Solving Differential Equations
(2012) Differential Equations and Dynamical Systems, 20 (1), pp. 67-76.
\bibitem{-11}  Bhrawy, A.H., Alofi, A.S.
A Jacobi-Gauss collocation method for solving nonlinear Lane-Emden type equations
(2012) Communications in Nonlinear Science and Numerical Simulation, 17 (1), pp. 62-70.
\bibitem{-12}  Caglar, S.H., Ucar, M.F.
Non-polynomial spline method for a time-dependent heat-like Lane-Emden equation
(2012) Acta Physica Polonica A, 121 (1), pp. 262-264.
\bibitem{-13} Shen, Q.
A meshless scaling iterative algorithm based on compactly supported radial basis functions for the numerical solution of Lane-Emden-Fowler equation
(2012) Numerical Methods for Partial Differential Equations, 28 (2), pp. 554-572.
\bibitem{-14} Pandey, R.K., Kumar, N., Bhardwaj, A., Dutta, G.
Solution of Lane-Emden type equations using Legendre operational matrix of differentiation
(2012) Applied Mathematics and Computation, 218 (14), pp. 7629-7637.
\bibitem{-15} Pandey, R.K., Kumar, N.
Solution of Lane-Emden type equations using Bernstein operational matrix of differentiation
(2012) New Astronomy, 17 (3), pp. 303-308.
\bibitem{-16} Herbst, R.S., Momoniat, E.
On computing the structure of critically rotating white dwarfs
(2012) New Astronomy, 17 (4), pp. 388-391.
\bibitem{-17}  Biezuner, R.J., Brown, J., Ercole, G., Martins, E.M.
Computing the first eigenpair of the p-Laplacian via inverse iteration of sublinear supersolutions
(2012) Journal of Scientific Computing, 52 (1), pp. 180-201.
\bibitem{-18}  Boubaker, K., Van Gorder, R.A.
Application of the BPES to Lane-Emden equations governing polytropic and isothermal gas spheres
(2012) New Astronomy, 17 (6), pp. 565-569.
\bibitem{-19}  Rismani, A.M., Monfared, H.
Numerical solution of singular IVPs of Lane-Emden type using a modified Legendre-spectral method
(2012) Applied Mathematical Modelling, 36 (10), pp. 4830-4836.
\bibitem{-20}  Jalab, H.A., Ibrahim, R.W., Murad, S.A., Melhum, A.I., Hadid, S.B.
Numerical solution of Lane-Emden equation using neural network
(2012) AIP Conference Proceedings, 1482, pp. 414-418.
\bibitem{-21}  Wazwaz, A.-M., Rach, R., Duan, J.-S.
Adomian decomposition method for solving the Volterra integral form of the Lane-Emden equations with initial values and boundary conditions
(2013) Applied Mathematics and Computation, 219 (10), pp. 5004-5019.
\bibitem{-22}  Karimi Vanani, S., Soleymani, F.
An extension of the Tau numerical algorithm for the solution of linear and nonlinear Lane-Emden equations
(2013) Mathematical Methods in the Applied Sciences, 36 (6), pp. 674-682.
\bibitem{-23}  Parand, K., Nikarya, M., Rad, J.A.
Solving non-linear Lane-Emden type equations using Bessel orthogonal functions collocation method
(2013) Celestial Mechanics and Dynamical Astronomy, 116 (1), pp. 97-107.
\bibitem{-24}  Huré, J.-M.
Exact, singularity-free recasting of the Newtonian potential in continuous media
(2013) Astronomy and Astrophysics, 554, art. no. A45, .
\bibitem{-25}  Kaur, H., Mittal, R.C., Mishra, V.
Haar wavelet approximate solutions for the generalized Lane-Emden equations arising in astrophysics
(2013) Computer Physics Communications, 184 (9), pp. 2169-2177.
\bibitem{-26}  \"{O}zt\"{u}rk, Y., G\"{u}lsu, M.
An operational matrix method for solving lane-emden equations arising in astrophysics
(2013) Mathematical Methods in the Applied Sciences, . Article in Press.
\bibitem{-27}  Razzaghi, M.
Hybrid functions for nonlinear differential equations with applications to physical problems
(2013) Lecture Notes in Computer Science (including subseries Lecture Notes in Artificial Intelligence and Lecture Notes in Bioinformatics), 8236 LNCS, pp. 86-94.
\bibitem{-27}  Razzaghi, M.
Hybrid functions for nonlinear differential equations with applications to physical problems
(2013) Lecture Notes in Computer Science (including subseries Lecture Notes in Artificial Intelligence and Lecture Notes in
Bioinformatics), 8236 LNCS, pp. 86-94.

\bibitem{-28} Li, Z.-X., Wang, Z.-Q.
Pseudospectral methods for computing the multiple solutions of the Lane-Emden equation
(2013) 255, pp. 407-421.

\bibitem{-29}  Farinelli, R., De Laurentis, M., Capozziello, S., Odintsov, S.D.
Numerical solutions of the modified Lane-Emden Equation in f(R)-gravity
(2014) Monthly Notices of the Royal Astronomical Society, 440 (3), pp. 2894-2900.
\bibitem{-30}  Wazwaz, A.-M.
The variational iteration method for solving the Volterra integro-differential forms of the Lane-Emden equations of the first and the second kind
(2014) Journal of Mathematical Chemistry, 52 (2), pp. 613-626.
\bibitem{-31} Rach, R., Duan, J.-S., Wazwaz, A.-M.
Solving coupled Lane-Emden boundary value problems in catalytic diffusion reactions by the Adomian decomposition method
(2014) Journal of Mathematical Chemistry, 52 (1), pp. 255-267.
\bibitem{-32}  Mohammadzadeh, R., Lakestani, M., Dehghan, M.
Collocation method for the numerical solutions of Lane-Emden type equations using cubic Hermite spline functions
(2014) Mathematical Methods in the Applied Sciences, 37 (9), pp. 1303-1317.
\bibitem{-33} Wazwaz, A.-M., Rach, R., Duan, J.-S.
A study on the systems of the Volterra integral forms of the Lane-Emden equations by the Adomian decomposition method
(2014) Mathematical Methods in the Applied Sciences, 37 (1), pp. 10-19.
\bibitem{-34}  \"{O}zt\"{u}rk, Y., G\"{u}lsu, M.
An approximation algorithm for the solution of the Lane-Emden type equations arising in astrophysics and engineering using Hermite polynomials
(2014) Computational and Applied Mathematics, 33 (1), pp. 131-145.
\bibitem{-35} Rach, R., Wazwaz, A.-M., Duan, J.-S.
The Volterra integral form of the Lane-Emden equation: new derivations and solution by the Adomian decomposition method
(2014) Journal of Applied Mathematics and Computing, . Article in Press.
\bibitem{-36} G\"{u}rb\"{u}z, B., Sezer, M.
Laguerre polynomial approach for solving Lane-Emden type functional differential equations
(2014) Applied Mathematics and Computation, 242, pp. 255-264.

\bibitem{-37} Mason J.C., Handscomb D.C. Chebyshev polynomials (2003) Chapman and Hall/CRC A CRC Press Company Boca Raton London New York Washington, D.C.

\bibitem{-38}G.P. Horedt, Polytropes: Applications in Astrophysics and Related Fields, Kluwer Academic Publishers, Dordrecht, 2004.
\bibitem{-39} A. Aslanov, A generalization of the Lane-Emden equation, Int. J. Comput.Math. 85 (2008) 1709-1725.
\bibitem{-40} A.S. Bataineh, M.S.M. Noorani, I. Hashim, Homotopy analysis method for singular IVPs of Emden-Fowler type, Commun. Nonlinear Sci. Numer. Simul. 14 (2009) 1121-1131.



\end{thebibliography}
\end{document}